\documentclass[oneside]{amsart}
\usepackage{geometry}
\usepackage{microtype} 
\usepackage{tikz}
\usetikzlibrary{arrows,calc,matrix,topaths,positioning,scopes,shapes,decorations,decorations.markings,decorations.pathmorphing} 
\usepackage{tkz-euclide}
\usetkzobj{all}
\usepackage{stmaryrd}
\usepackage{verbatim}
\usepackage{float}
\usepackage{textcomp}
\usepackage{amscd,amsfonts}
\usepackage{amsthm,amsmath}
\usepackage{amstext,amssymb}
\usepackage{xcolor}
\usepackage{graphicx}
\usepackage{esint}
\usepackage[hidelinks]{hyperref}
\usepackage{marginnote}
\usepackage[all]{xy}
\reversemarginpar
\usepackage{enumerate}
\usepackage{enumitem}
\newlist{enumD}{enumerate}{1}
\setlist[enumD]{label=(D\arabic*)}
\newlist{enumPL}{enumerate}{1}
\setlist[enumPL]{label=(PL\arabic*)}
\newlist{enumR}{enumerate}{1}
\setlist[enumR]{label=(R\arabic*)}
\newlist{enumEQ}{enumerate}{1}
\setlist[enumEQ]{label=(EQ\arabic*)}
\newlist{enumM}{enumerate}{1}
\setlist[enumM]{label=(M\arabic*)}
\newlist{enumRe}{enumerate}{1}
\setlist[enumRe]{label=(R\arabic*)}

\numberwithin{equation}{section}
\numberwithin{figure}{section}
\theoremstyle{plain}
\newtheorem{thm}{\protect\theoremname}
\theoremstyle{definition}
\newtheorem{defn}[thm]{\protect\definitionname}
\theoremstyle{plain}
\newtheorem{prop}[thm]{\protect\propositionname}
\theoremstyle{remark}
\newtheorem{rem}[thm]{\protect\remarkname}
\theoremstyle{definition}
\newtheorem{example}[thm]{\protect\examplename}
\theoremstyle{plain}
\newtheorem{lem}[thm]{\protect\lemmaname}
\theoremstyle{plain}

\theoremstyle{definition}
\newtheorem{notation}[thm]{\protect\notationname}

\providecommand{\corollaryname}{Corollary}
\providecommand{\definitionname}{Definition}
\providecommand{\examplename}{Example}
\providecommand{\lemmaname}{Lemma}
\providecommand{\propositionname}{Proposition}
\providecommand{\remarkname}{Remark}
\providecommand{\theoremname}{Theorem}
\providecommand{\notationname}{Notation}

\def\g{{\mathfrak g}}

\def\shuff{\operatorname{sh}}
\def\h{{\mathfrak h}}
\def\Hom{\operatorname{Hom}}
\def\Cross{\operatorname{Cross}}
\def\Der{\operatorname{Der}}
\def\End{\operatorname{End}}
\def\Aut{\operatorname{Aut}}
\def\CM{\cat{CM}}
\def\CMinv{\cat{CM_{inv}}}
\def\CMid{\cat{CM_{id}}}
\def\CMGp{\cat{CMGp}}
\def\CMGpinv{\cat{CMGp_{inv}}}
\def\CMGpid{\cat{CMGp_{id}}}

\def\PB{\cat{Pbialg}}
\def\IsoPB{\cat{Pbialg_{inv}}}

\def\BCH{\operatorname{BCH}}
\def\locGp{\cat{CMGp^{loc}}}
\def\Exp{\operatorname{Exp}}
\def\ImU{\cat{Im}(\widehat{\ca{U}})} 
\def\CMinvfin{\cat{CM^{fin}_{inv}}} 
\def\PBcomp{\cat{\widehat{Pbialg}_{inv}}}

\newcommand{\ie}{\emph{i.e.} }
\newcommand{\ot}{\otimes}
\newcommand{\cat}[1]{\mathbf{#1}}
\newcommand{\ca}[1]{\mathcal{#1}}
\def\co{\colon\thinspace}
\newcommand{\PostLie}{\mathcal{P}ost\mathcal{L}ie}
\newcommand{\plprod}{\triangleright}

\newcommand{\K}{\mathbb{K}}

\newcommand{\TB}{\ca{PSB}}
\newcommand{\PSB}{\ca{PSB}}
\newcommand{\Forest}{\ca{F}or}
\newcommand{\bsq}{\begin{tikzpicture}[scale=1,baseline=-2.5pt]
	\draw [black,fill=black] (-.06,-.06) rectangle (.06,.06); 
	\end{tikzpicture}}
\newcommand{\bsqsmall}{\begin{tikzpicture}[scale=.7,baseline=-2pt]
	\draw [black,fill=black] (-.06,-.06) rectangle (.06,.06); 
	\end{tikzpicture}}

\newcommand{\id}{\operatorname{id}}

\tikzstyle{my circle}=[draw, fill, circle, minimum size=3pt, inner sep=0pt]	
\tikzstyle{my circle bis}=[draw, fill, circle, minimum size=3.5pt, inner sep=0pt]	
\tikzstyle{sq}=[draw, fill, rectangle, minimum size=3pt, inner sep=0pt]		

\begin{document}
\title[Crossed morphisms, post-Lie algebras and the post-Lie Magnus expansion]{Crossed morphisms, (integration of) post-Lie algebras and the post-Lie Magnus expansion}

\author{Igor Mencattini}
\address{Universidade de S\~ao Paulo,
	Instituto de Ci\^encias Matem\'aticas e de Computa\c c\~ao,
	Avenida Trabalhador S\~ao-carlense, 400 -
	CEP: 13566-590 - S\~ao Carlos, SP -
	Brazil}
\email{igorre@icmc.usp.br}

\author{Alexandre Quesney}
\address{Universidade de S\~ao Paulo,
	Instituto de Ci\^encias Matem\'aticas e de Computa\c c\~ao,
	Avenida Trabalhador S\~ao-carlense, 400 -
	CEP: 13566-590 - S\~ao Carlos, SP -
	Brazil}
\email{alexquesney.math@noyta.net}


\date{\today}

\begin{abstract}
This letter is divided in two parts. In the first one it will
be shown that the datum of a post-Lie product is equivalent to the one of an invertible crossed
morphism between two Lie algebras. Moreover it will be argued that the integration of such
a crossed morphism yields the post-Lie Magnus expansion associated to the original post-Lie
algebra. The second part is devoted to present two combinatorial methods to compute the coefficients of
this remarkable formal series. Both methods are based on special tubings on planar trees.	
\end{abstract}

\maketitle
\tableofcontents

\section{Introduction}
Pre and post-Lie algebras are two classes of non-associative algebras which later on have undergone to an extensive investigation because of the important role they play both in pure an applied mathematics. 
Pre-Lie algebras, also known in the literature under the name of left-symmetric and Vinberg algebras, were introduced in the mid-sixty, almost simultaneously, by Gerstenhaber, see \cite{Gersten}, and Vinberg, see \cite{Vinberg}, the first working on the theory of deformation of associative algebras and second on the theory of convex cones. Since then, they appeared unexpectedly in almost every area of the modern mathematics, from differential geometry, \cite{medina, hessian, bandiera} to combinatorics \cite{BS, CL, CP}, from mathematical physics, see \cite{EFP}, to numerical analysis \cite{Brouder, HLW}, see \cite{burderev,Manchon,EFP1} for comprehensive reviews.
In spite post-Lie algebras have been introduced much more recently by Vallette, see \cite{Val}, and independently by Lundervold and Munthe-Kaas \cite{MKL}, since then they have been deeply studied, both from point of view of pure, see for example \cite{bai-guo-ni, bdv, EFMM, MQS} and of applied mathematics \cite{EFLMMK,MKSV,CEFMK}, see also \cite{KI,F-MK}. 

Recall that a post-Lie Lie algebra is a pair $(\h,\triangleright)$ of a Lie algebra $\h$, whose Lie bracket will be denoted by $[-,-]$, and a bilinear  map $\triangleright:\h\otimes\h\rightarrow\h$ called post-Lie product, satisfying the following two properties:
\begin{enumPL}
\item\label{PL Property 1} $x\triangleright [y,z]=[x\triangleright y,z]+[y,x\triangleright z]$, and
\item\label{PL Property 2} $[x,y]\triangleright z={\rm a}_{\triangleright}(x,y,z)-{\rm a}_{\triangleright}(y,x,z)$, for all $x,y$ and $z$ in $\h$.
\end{enumPL}
In the RHS of \ref{PL Property 2} 
\begin{equation}
{\rm a}_\triangleright(x,y,z)=x\triangleright(y\triangleright z)-(x\triangleright y)\triangleright z,\,\forall x,y,z\in\mathfrak h \label{eq:ass}
\end{equation}
denotes the \emph{associator} defined by the bilinear product $\triangleright$. 
A post-Lie algebra whose Lie bracket is trivial is a pre-Lie algebra.  
On the other hand, a post-Lie algebra $(\h,\triangleright)$ gives rise to a Lie algebra $\overline{\h}$ with same underlying vector space as $\h$ and whose Lie bracket $\llbracket-,-\rrbracket:\h\otimes \h\rightarrow \h$ is defined by
\begin{equation}
\llbracket x,y\rrbracket=x\triangleright y-y\triangleright x+[x,y],\,\forall x,y\in\h. \label{eq:classpostLie}
\end{equation}
Moreover, there is an element $\upsilon \in\Hom_{\text{Lie}}(\overline{\h},\Der(\h))$, defined by 
\begin{equation}
\upsilon_x(y)=x\triangleright y,\,\forall x,y\in\overline{\h}. 
\end{equation} 

The enveloping algebra $\mathcal U(\h)$ of a post-Lie algebra was analyzed in depth in \cite{EFLMMK}, whose authors, extending the results of \cite{GO1,GO2}, showed that a suitable extension of $\triangleright$ together with the coalgebra structure of $\mathcal U(\h)$ allow to define a \emph{new} associative product $\ast\co \mathcal U(\h)\otimes\mathcal U(\h)\rightarrow\mathcal U(\h)$, called the \emph{Grossman-Larson} product, compatible with the initial coalgebra structure and antipode. 
In this way it was proven that on $\mathcal U(\h)$ it was possible to define a new Hopf algebra $\mathcal U_\ast(\h)$, which turned out to be isomorphic to $\mathcal U(\overline{\h})$.
After suitable completion of the Hopf algebras involved, the above mentioned isomorphism defines an isomorphism between the (completed) Lie algebras $\overline{\h}$ and $\h$, whose inverse $\chi \co \h\rightarrow\overline{\h}$, called the \emph{post-Lie Magnus expansion}, abbreviated as pLMe hereafter, is one of the main concerns of the present note. 

The pLMe has two predecessors, the pre-Lie and the classical Magnus expansions, see \cite{BCOR}.  
The \emph{pre}-Lie Magnus expansion $\chi$ appeared at the beginning of the eighties in the work of Agrachev and Gramkelidze, see \cite{AG}. 
However it has been dubbed as such only in \cite{Kur-Man}, where the classical Magnus expansion was extensively explored in the context of the pre-Lie and dendriform algebras. 
Finally in \cite{CP}  was presented a formula expressing $\chi$ in terms of the so called Grossman-Larson product which read as
\[
\chi(x)=\log_\ast(\exp(x)),
\]
see also \cite{BS}. 

On the other hand, the pLMe was introduced in \cite{EFLMMK} in connection with a particular class of iso-spectral flow equations. There it is was shown that for every $x\in\h$,  $\chi_x(t):=\chi(tx)\in\h[[t]]$ satisfies the following non-linear ODE
\[
\dot{\chi}_x(t)=(d\exp_{\ast})^{-1}_{-\chi_x(t)}\big(\exp_{\ast}(-\chi_x (t))\triangleright x\big),
\]
and that, the \emph{non-linear post-Lie differential} equation 
\[
{\dot x}(t)=-x(t)\triangleright x(t),
\]
for $x=x(t)\in\h[[t]]$, with initial condition $x(0)=x_0\in\h$, has as a solution
\[
x(t)=\exp_\ast (-\chi_{x_0}(t))\triangleright x_0.
\]
In \cite{CEFO} it was underlined the relevance of the pLMe in the theory of the Lie group integrators and in \cite{MQS} it was proven that on a post-Lie algebra, in analogy to what happens on every pre-Lie algebra, the pLMe provides an isomorphism between the group of \emph{formal flows} and the BCH-group defined on $\overline{\h}$, generalizing the analog well known result proven in \cite{AG}, see also \cite{DSV,bandiera}. 
\\

The aim of this letter is twofold. In the first place, starting from the integration result presented in \cite{MQS}, we give a more Lie-theoretic interpretation of the pLMe, in terms of the so called \emph{crossed morphisms} of Lie groups and Lie algebras, see for example \cite{Lue, B-N, Higert-Neeb, PSTZ}. 

More precisely, first we show that a post-Lie algebra structure on a Lie algebra $\h$ is equivalent to the datum $(\id,\upsilon)$ where $\id \co \overline{\h}\rightarrow\h$, the identity map, is a crossed morphism relative to $\upsilon\in\Hom_{\text{Lie}}(\overline{\h},\Der(\h))$. Then we argue that the pLMe is the (inverse of a) crossed morphism between the corresponding (local) Lie groups $\overline{\mathcal H}$ and $\mathcal H$, obtained \emph{integrating} $\id$. 

We would like to stress that while the local existence of the pLMe, at the level of the Lie groups $H$ and $\overline H$, is guaranteed by general Lie theory, its global existence is obstructed, see \cite{NeebLoc}, and its explicit expression seems, from this view-point, really difficult to obtain. 

On the other hand, working formally at the level of the completed enveloping algebras of $\overline{\h}$ and $\h$, one first can prove the existence of the pLMe and then, using a \emph{formal} integration process, can show that the inverse of the pLMe is a crossed morphism between the corresponding local Lie groups.
 
This line of thoughts opens the door to a categorical interpretation of various approaches to post-Lie and pre-Lie algebras which one finds in the literature, see for example \cite{bai-guo-ni,bdv,MKL} and references there in.
On the other hand, it makes clear the universal nature of the pLMe, asking for a (more systematic) method to compute the coefficients of this expansion.
This goal is achieved in the second part of the paper, where such a method, based in the so called \emph{tubings}, see \cite{CarrDevadoss}, is presented.
\\\\
\emph{Relations with other works.} Post-Lie algebras appeared recently as central objects in the study of the so called \emph{$\mathcal{O}$-operators}, first introduced in \cite{Kup}, which are particular extensions of the classical $r$-matrices, playing an important role in the theory of the generalized Lax pair representations, introduced in \cite{Bordemann}. The notion of $\mathcal O$-operator was further extended in \cite{bai-guo-ni}, where the concepts of $\mathcal O$-operator of \emph{weight $\lambda$} and, respectively, of \emph{extended} $\mathcal O$-operator, were introduced. It is in this framework that the relation between (generalized) Lax representations and post-Lie algebras crystallized. In particular in \cite{bai-guo-ni} it was shown that the post-Lie algebra structure on a Lie algebra $\g$ are in one-to-one correspondence with the pairs $((\upsilon,\h),\mathcal O)$ where $\h$ is a Lie algebra, $\upsilon\in\Hom_{\text{Lie}}(\g,\Der(\h))$ and $\mathcal O:\h\rightarrow\g$ is an \emph{invertible} $\mathcal O$-operator of weight $1$, see Corollary 5.5 in \cite{bai-guo-ni}. This result should be compared with Proposition \ref{pro:equiv} of the present work, see also \ref{Re1} in Remark \ref{rem:relationwithothers}.

Another instance where the notion of a post-Lie algebra rises naturally is the theory of the so called simply-transitive \emph{NIL-affine actions} of nilpotent Lie groups, see \cite{bdv}. 
In this reference it was shown that given $(G,N)$, a pair of connected and simply-connected nilpotent Lie groups, there exists a simply transitive NIL-action of $G$ on $N$ if and only if there exists a Lie algebra $\g'\sim \g$ such that the pair $(\g',\mathfrak{n})$ carries a structure of a post-Lie algebra, see Theorem 2.5 in \cite{bdv}. 
The proof of this result is based on the observation that a pair of Lie algebras $(\g,\mathfrak n)$ carries a structure of a post-Lie algebra if and only if there is a faithful morphism of Lie algebras $\varrho:\g\rightarrow\mathfrak n\ltimes\Der(\mathfrak n)$ of the form $\varrho(x)=(x,L(x))$ for all $x\in\g$, see Proposition 2.11 in \cite{bdv} for the precise statement. This result should be compared with \ref{Re2} in Remark \ref{rem:relationwithothers} of the present work.
\\

\emph{Plan of the present work}.  
In Section \ref{sec: CM} is recalled the notion of crossed morphism for Lie groups and Lie algebras and the relation between \emph{invertible} crossed morphisms of Lie algebras and post-Lie algebras is explained. 
This section closes with a brief discussion on \emph{local} Lie groups. \\ 
In Section \ref{sec: Univ. env. alg. and pLMe} it is shown that the datum of a crossed morphism between two Lie algebras yields a morphism between the associated universal enveloping algebras, which, when the crossed morphism is invertible, provides an isomorphism giving rise to the \emph{Grossman-Larson} product. 
After introducing a suitable \emph{integration functor}, last part of this section is devoted to the analysis of the pLMe from the categorical view-point sketched above. 
\\
In Section \ref{sec: Computing pLMe} two combinatorial interpretations of the coefficients of the pLMe are given. 
Both interpretations are based on a notion of nested tubings. 
The first method is based on the \emph{vertical} nested tubings and allows to compute the coefficients associated to any forest recursively. 
The second method is based on the \emph{horizontal} nested tubings and allows to express these coefficients in a closed form. 
This section is divided into six parts. 
The first four parts \ref{sec: Operad of PRT}--\ref{sec: univ env alg} are essentially a reminder; they serve to set up conventions and to introduce the adequate combinatorics in order to handle the pLMe. 
More in details, the first part set up conventions and notations on planar trees and forests and introduces specific graftings of them. 
The second part is a brief reminder on the combinatorial operad  $\PSB$, a model of the operad $\ca{P}ost\ca{L}ie$, which serves as a combinatorial base to handle operations on the free post-Lie algebra (on one generator) and on its universal enveloping algebra. 
These last two algebras are the subject of the third and four parts.  
The fifth part is dedicated to the notions of vertical and horizontal nested tubings which are the last essential ingredient to compute the pLMe. 
Finally, the last part is devoted to the computation of this expansion, first in terms of vertical, then in terms of horizontal nested tubings.

 \subsection{Conventions}\label{sec: convention operads} 

Throughout the paper $\mathbb K$ will denote a field of characteristic zero.  
 The tensor product will be taken over $\mathbb K$. In particular the tensor product of two $\mathbb K$--vector spaces $V$ and $W$ will be denoted by $V\ot W$. 
 All Lie groups considered will be connected and simply-connected. 
The category of the post-Lie algebras and their morphisms will be denoted by $\cat{PostLie}$ while the category of the pre-Lie algebras and their morphisms will be denoted by $\cat{PreLie}$.

\section{Crossed morphisms}\label{sec: CM}

In this section we will recall the concepts of post-Lie algebra and of crossed morphism, both for Lie algebras and for Lie groups, and we will comment on how these relate to each other.

\subsection{Crossed morphisms of Lie algebras}

\begin{defn}
	Let $\g$ and $\h$ be two Lie algebras, and let $\upsilon:\g\rightarrow\Der_{\text{Lie}}(\h)$ be a morphism of Lie algebras. 
	A \emph{crossed morphism relative to $\upsilon$} is a map $\phi\in\Hom_{\mathbb K}(\mathfrak g,\mathfrak h)$ that satisfies 
	\begin{equation*}
	\phi([x,y]_{\mathfrak g})=\upsilon_x(\phi(y))-\upsilon_y(\phi(x))+[\phi(x),\phi(y)]_{\mathfrak h},\,\forall x,y\in\mathfrak g.\label{eq:crosshom}
	\end{equation*} 
	The set of crossed morphisms of $\mathfrak g$ in $\mathfrak h$ relative to $\upsilon$ is denoted with $\operatorname{Cross}^\upsilon(\mathfrak g,\mathfrak h)$. 
	The subset of the \emph{invertible} crossed morphisms is denoted with $\Cross_{\text{inv}}^{\upsilon}(\g,\h)$. 
\end{defn}

\begin{example}
	If $\mathfrak h$ is abelian, \ie if $[-,-]_{\mathfrak h}\equiv 0$, then $\operatorname{Der}_{\text{Lie}}(\mathfrak g)=\operatorname{End}_{\mathbb K}(\mathfrak g)$. In this case $\phi$ is a crossed morphism of $\mathfrak g$ in $\mathfrak h$ relative to $\upsilon\in\operatorname{End}_{\mathbb K}(\mathfrak g)$ if and only if 
	\[
	\phi([x,y]_{\mathfrak g})=\upsilon_x(\phi(y))-\upsilon_y(\phi(x)),\,\forall x,y\in\mathfrak g.
	\]
\end{example}
\begin{example}
	If $f\in\operatorname{Hom}_{\text{Lie}}(\mathfrak g,\mathfrak h)$ then $\upsilon_f:\mathfrak g\rightarrow\operatorname{Der}_{\text{Lie}}(\mathfrak h)$ defined by
	\[
	\upsilon_{f}(x)(a)=[f(x),a]_{\mathfrak h},\,\forall a\in\mathfrak h,
	\]
	is a morphism of Lie algebras and $\phi\in\operatorname{Hom}_{\mathbb K}(\mathfrak g,\mathfrak h)$ belongs to $\operatorname{Cross}^{\upsilon_f}(\mathfrak g,\mathfrak h)$ if and only if $f+\phi\in\operatorname{Hom}_{\text{Lie}}(\mathfrak g,\mathfrak h)$.
\end{example}

\begin{example}
	If $\g=\h$ and $\phi$ is the identity, then $\h$ has another Lie algebra structure, given by 
	\begin{equation}\label{eq: h bar }
	\llbracket x,y \rrbracket := \upsilon_{x}(y) - \upsilon_y (x) + [x,y]_{\h} \text{ for all } x,y \in \h. 
	\end{equation}
	The resulting Lie algebra is denoted by $\overline{\h} = (\h,\llbracket-,-\rrbracket)$. 
\end{example}

\begin{defn}
The category $\CM$ is as follows. 
The objects are the tuples $(\g,\h,\upsilon, \phi)$ of two Lie algebras $\g$ and $\h$ and 
$(\upsilon,\phi)\in \Hom_{\text{Lie}}(\g,\Der(\h)) \times \Cross^{\upsilon}(\g,\h)$. 
The morphisms between $(\g,\h,\upsilon, \phi)$ and $(\g',\h',\upsilon', \phi')$ are pairs $(f,g)\in\Hom_{\text{Lie}}(\g,\g')\times \Hom_{\text{Lie}}(\h,\h')$ such that
\begin{enumM}
\item\label{M1}$g\circ\phi=\phi'\circ f$ and
\item\label{M2}$g(\upsilon_{x}(a))=\upsilon'_{f(x)}(g(a)) \text{ for all } x\in\mathfrak g \text{ and } a\in\mathfrak h$. 
\end{enumM}

The subcategory $\CMinv\subset \CM$  is the one of those tuples $(\g,\h,\upsilon,\phi)$ such that $\phi$ is an invertible crossed morphism, \ie $\phi\in \Cross^{\upsilon}_{\text{inv}}(\g,\h)$.   
The subcategory $\iota\co \CMid\subset \CMinv$ is the one of the tuples of the form  $(\overline{\h},\h,\upsilon, id)$. 
\end{defn}

Let $R\co \CMinv \to \CMid$ be the functor given by 
\begin{equation}
R(\g,\h,\upsilon, \phi)=(\overline{\h},\h,\upsilon\circ \phi^{-1}, id)\quad\text{and}\quad R(f,g)=(g,g).\label{eq:functR}
\end{equation} 
Note that in the first tuple, the Lie algebra $\overline{\h}$ is determined by $(\h,\upsilon\circ \phi^{-1})$ so that is Lie bracket is given  $\llbracket x,y\rrbracket=\upsilon_{\phi^{-1}(x)}(y)-\upsilon_{\phi^{-1}(y)}(x)+[x,y]_{\h}$ for all $x,y\in\h$, according to \eqref{eq: h bar }.

\begin{prop}\label{pro:equivcat}
	The two categories $\CMid$ and $\CMinv$ are adjoint equivalent. 
\end{prop}
\begin{proof}
	The inclusion functor is full and essentially surjective which proves the equivalence. 
	It remains to show that the  functor  $R$ is a right adjoint to $\iota$. 
	To do this it is enough to check that the unit $\eta\co id \to R\circ\iota$ and counit $\epsilon\co \iota \circ R \to id$  transformations satisfy the triangle relations: $\iota \xrightarrow{\iota \eta} \iota R \iota \xrightarrow{\epsilon \iota} \iota$ and $R  \xrightarrow{\eta R} R\iota R  \xrightarrow{R \epsilon} R$ are identities. This is a straightforward verification.
\end{proof}

To a tuple $(\overline{\h},\h,\upsilon,id) \in \CMid$ one may associate the post-Lie algebra $(\h,\plprod)$ where 
\begin{equation}\label{eq: postlie from hh,v ,id}
x\plprod y := \upsilon_{x}(y) \text{ for all } x,y\in \h. 
\end{equation}
Indeed, \ref{PL Property 1} is clear since $\upsilon_x$ is a derivation  of $\h$, and  
\ref{PL Property 2} results from the fact that $\upsilon\co \overline{\h}\to \Der(\h)$ is a Lie morphism: for all $x,y$ and $z$ in $\h$, one has 
\begin{equation*}
[x,y]\plprod z
=\upsilon_{[x,y]}(z)
=\upsilon_{\llbracket x,y \rrbracket - \upsilon_{x}(y) + \upsilon_{y}(x)}(z)\\
=\upsilon_{x}(\upsilon_{y}(z))-\upsilon_{y}(\upsilon_{x}(z))-\upsilon_{\upsilon_{x}(y)}(z)+\upsilon_{\upsilon_{y}(x)}(z).
\end{equation*}
The following is straightforward. 
\begin{prop}\label{pro:equiv}
	The two categories  $\CMid$ and $\cat{PostLie}$ are isomorphic. 
\end{prop}  

\begin{rem}\label{ex:preLie}
The tuples $(\mathfrak g,\mathfrak h,\upsilon,\phi)$ as in $\CMinv$ where $\h$ is an abelian Lie algebra form a full subcategory of $\CMinv$, denoted, hereafter, by $\cat{CM_{pl}}$. 
The full subcategory of $\cat{CM_{pl}}$ whose objects are the tuples $(\overline{\h},\h,\upsilon,\id)$, denoted from now on by $\cat{CM_{pl,id}}$, which is adjoint equivalent to $\cat{CM_{pl}}$, is isomorphic to $\cat{PreLie}$, recovering the result of \cite{bai}, see also \cite{baus} and references therein.
To be more explicit, it is worth to note that if $(\mathfrak g,\mathfrak h,\upsilon,\phi)$ is an object in $\cat{CM_{pl}}$, then
\[
\phi([x,y]_\g)=\upsilon_x(\phi(y))-\upsilon_y(\phi(x)),\,\forall x,y\in\g,
\]
i.e. $\phi$ is a \emph{bijective} $1$-cocycle (of the Chevalley-Eilenberg cohomology) of $\g$ with values in $\h$. 
\end{rem}

\begin{example}[\cite{B-N},\cite{MKL}]
For a given Lie group $K$ whose Lie algebra is $\mathfrak k$, let $\g=\mathfrak X(K)$ with its standard Lie bracket and $\h=C^\infty(K,\mathfrak k)$ with the Lie bracket defined by $\llceil f,g\rrceil(k)=[f(k),g(k)]_\mathfrak k$, for all $f,g\in\h$ and $k\in K$. Then
\begin{equation}
\upsilon_X(f)(k):=(X_kf),\,\forall X\in\g,\;f\in\h,\label{eq:up}
\end{equation}
is a morphism of Lie algebras from $\g$ to $\Der(\h)$. 
Furthermore, recall that $\theta\in\Omega^1(K,\mathfrak k)$, defined via the left-translations $L_k$ by $\theta_k(v)=(L_{k^{-1}})_{\ast,k}(v)$ for all $k\in K$ and $v\in T_kK$, defines a parallelization of $TK{\simeq}K\times\mathfrak k$ by the $v\stackrel{\theta}{\rightsquigarrow}(k,\theta_k(v))$, for all $k\in K$ and $v\in T_kK$. Composing this map with the projection $K\times\mathfrak k\rightarrow\mathfrak k$, one obtains $\phi\in C^\infty(K,\mathfrak k)$
\begin{equation}
\phi(X)=i_X\theta,\,\forall X\in\g.\label{eq:phi}
\end{equation}
Computing $i_{[X,Y]}\theta=\mathcal L_X(i_Y\theta)-i_Y(\mathcal L_X\theta)$, where $\mathcal L_X$ denotes the operation of Lie derivative in the direction $X$, and recalling that $\theta$ satisfies the Maurer-Cartan equation, i.e. $d\theta+\frac{1}{2}[\theta,\theta]=0$, one obtains
\[
i_{[X,Y]}\theta=i_X(di_Y\theta)-i_Y(di_X\theta)+\llceil i_X\theta,i_Y\theta\rrceil,\,\forall X,Y\in\g,
\]
i.e. $\phi\in\Cross^\upsilon_{\text{inv}}(\g,\h)$. The application $\phi$ is an invertible map, $C^\infty(K)$-linear, such that $\phi(X_x)=x$ for all $x\in\mathfrak k$, where $X_x$ is the left-invariant vector field defined by the element $x\in\mathfrak k$. Applying the functor $R$ defined in \eqref{eq:functR}, one concludes that
\begin{equation}
f\triangleright g:=\phi^{-1}(f)g\,\forall f,g\in\h,\label{eq:Kpl}
\end{equation}
makes $(\h, \llceil-,-\rrceil, \triangleright)$ into a post-Lie algebra. 

Moreover, the Lie algebra $\overline\h$ whose underlying vector space is $C^\infty(K,\mathfrak k)$ and whose Lie bracket is
\[
\llfloor f,g\rrfloor=\phi^{-1}(f)g-\phi^{-1}(g)f+\llceil f,g\rrceil,\,\forall f,g\in C^\infty(K,\mathfrak k).
\]
Pulling back \eqref{eq:Kpl} to $\g$, one obtains $\blacktriangleright:\g\otimes\g\rightarrow\g$, defined by
\begin{equation}
X\blacktriangleright Y=\phi^{-1}(\phi(X)\triangleright\phi(Y)),\,\forall X,Y\in\g,
\end{equation}
which is a $C^\infty(K)$-linear product on $\mathfrak X(K)$ with respect to the first entry, such that
\[
X\blacktriangleright(\xi Y)=X(\xi)Y+\xi X\blacktriangleright X,\,\forall\xi\in C^\infty(K),\,X,Y\in\mathfrak X(K),
\]
which, together with \eqref{eq:Kpl}, implies that $X\blacktriangleright Y=0$ for all $X\in\mathfrak X(K)$ and all $Y$ left-invariant.
In other words, $\blacktriangleright$ defines a flat linear connection on $TK$, whose flat sections are the left-invariant vector fields, and whose torsion is easily shown to be parallel since $T(X_x,X_y)=-X_{[x,y]_\mathfrak t}$, for all $x,y\in\mathfrak k$.
\end{example}

\begin{rem}\label{rem:relationwithothers}
A couple of remarks are now in order.
\begin{enumerate}
\item\label{Re1} Keeping the same notations introduced above, $r\in\Hom_{\K}(\h,\g)$ is called an $\mathcal O$-operator of weight $\lambda\in\mathbb R$, if
\begin{equation}
[r(x),r(y)]_{\g}=r(\upsilon_{r(x)}y-\upsilon_{r(y)}x+\lambda[x,y]_\h).\,\forall x,y\in\h.\label{eq:Ope}
\end{equation}
The tuple $(\g,\h,\upsilon,r)$ where $r$ satisfies \eqref{eq:Ope} form a category $\CM_{\mathcal O,\lambda}$ whose morphisms between $(\g,\h,\upsilon,r)$ and $(\g',\h',\upsilon',r')$ are pairs $(f,g)\in\Hom_{\text{Lie}}(\g,\g')\times\Hom_{\text{Lie}}(\h,\h')$, satisfying \ref{M2} and the analogue of \ref{M1}, i.e. $f\circ r=r'\circ g$.
The full subcategory of $\CM_{\mathcal O,\lambda=1}$ whose objects are the tuples $(\g,\h,\upsilon,r)$ whose $r$ is invertible is isomorphic to $\CMinv$, because of Proposition \ref{pro:equiv}, it is adjoint equivalent to $\cat{PostLie}$. In this way we recover the description of post-Lie algebras given in \cite{bai-guo-ni}. 
\item\label{Re2} Let $\CM_b$ be the category whose objects are the tuples $(\g,\h,\varrho)$, where $\varrho\in\Hom_{\text{Lie}}(\g,\h\rtimes\Der(\h))$ and the Lie bracket in 
$\h\rtimes\Der(\h))$ is defined by the formula 
\[
\{(h_1,d_1),(h_2,d_2)\}=([h_1,h_2]_{\h}+d_1(h_2)-d_2(h_1),[d_1,d_2]).
\]
Note that composing $\varrho:\g\rightarrow\h\rtimes\Der(\h)$ with the canonical projections $\pi_2:\h\rtimes\Der(\h)\rightarrow\Der(\h)$ and $\pi_1:\h\rtimes\Der(\h)\rightarrow\h$, one gets $\upsilon_\varrho\in\Hom_{\text{Lie}}(\g,\Der(\h))$ and, respectively, $\phi_\varrho\in\Cross^{\upsilon_\varrho}(\g,\h)$. 
A morphism between two objects $(\g,\h,\varrho)$ and $(\g',\h',\varrho')$ in $\CM_b$ is a pair $(f,g)\in\Hom_{\text{Lie}}(\g,\g')\times\Hom_{\text{Lie}}(\h,\h')$ satisfying \ref{M1} and \ref{M2} with respect to the pairs $(\upsilon_\varrho,\phi_\varrho)$'s. The full subcategory $\cat{CM_{b,inv}}\subset\CM_b$ whose objects are $(\g,\h,\varrho)$, 
where $\phi_\varrho$ is a bijective linear map, is easily shown to be isomorphic to $\CMinv$. 
Analogously the full subcategory $\cat{CM_{b,id}}\subset\cat{CM_{b,inv}}$ whose objects are $(\g,\h,\varrho)$ where $\g$ and $\h$ are defined on the same underlying vector space and $\phi_\varrho=\operatorname{id}$ turns out to be isomorphic to $\CMid$. In this way one recovers the description of $\cat{PostLie}$ given in \cite{bdv}. 
\end{enumerate}
\end{rem}

\subsection{Crossed morphism of Lie group type objects}

In analogy to the Lie algebra case one can define the notion of crossed morphism between two Lie groups. 

First, recall that for $H$ a Lie group, $\operatorname{Aut}(H)$ denotes the group of automorphisms of $H$ which are diffeomorphisms of $H$, i.e. $\phi\in\operatorname{Aut}(H)$ if and only if
\begin{enumerate}
	\item[(i)] $\phi$ is an isomorphism of abstract groups,
	\item[(ii)] $\phi$ is a diffeomorphism.
\end{enumerate}

\begin{defn}
	Let $G$ and $H$ be two Lie groups and let $\Upsilon\co G\rightarrow\Aut(H)$ be a morphism of Lie groups. 
	A \emph{crossed morphism relative to $\Upsilon$} is a smooth map 
	$\Phi\co G\rightarrow H$ that satisfies 
	\begin{equation}
	\Phi(gh)=\Phi(g)\Upsilon_g(\Phi(h)),\,\forall g,h\in G.\label{eq:crossG}
	\end{equation}
\end{defn}

By changing, in the definition of $\CM$, the underlying category of Lie algebras by the one of Lie groups, one obtains the following category. 
\begin{defn}
	The category $\CMGp$ is as follows. 
	The objects are the tuples $(G,H,\Upsilon, \Phi)$ of two Lie groups $G$ and $H$ and 
	$(\Upsilon,\Phi)\in \Hom_{\text{LieGp}}(G,\Aut(H)) \times \Cross^{\Upsilon}(G,H)$. 
	The morphisms between $(G,H,\Upsilon, \Phi)$ and $(G',H',\Upsilon', \Phi')$ are pairs 
	$(f,g)\in\Hom_{\text{LieGp}}(G,G')\times \Hom_{\text{LieGp}}(H,H')$ such that
	\begin{equation*}
	g\circ\Phi=\Phi'\circ f \quad  \text{ and } \quad
	g(\Upsilon_{x}(a))=\Upsilon'_{f(x)}(g(a)) \text{ for all } x\in G \text{ and } a\in H.
	\end{equation*}
\end{defn} 

In the same vein as before, one has subcategories $\CMGpid \subset \CMGpinv \subset \CMGp$, and an adjoint equivalence between $\CMGpid$ and $\CMGpinv$. 
Note that the projection functor 
\begin{equation*}
P\co \CMGpinv\to \CMGpid
\end{equation*} 
sends any tuple $(G,H,\Upsilon,\Phi)$ to $(\overline{H}, H, \Upsilon\circ \Phi^{-1},id)$, where 
the product of $\overline{H}=(H,\star)$ is given by 
\begin{equation}\label{eq: prod H bar}
h_1\star h_2 = h_1\Upsilon_{\Phi^{-1}(h_1)}(h_2) \text{ for all } h_1,h_2\in H. 
\end{equation}

The classical Lie functor gives rise to a functor 
\begin{equation*}
T_e\co \cat{CMGp} \to \CM
\end{equation*}
that sends $(G,H,\Upsilon, \Phi)$ to $(\g,\h,\Upsilon_{\ast,e_G}, \Phi_{\ast,e_G})$. 
It restricts to the sub categories of invertible crossed morphisms  
\begin{equation*}
T_e\co \cat{CMGp_{inv}} \to \CMinv  
\end{equation*} 
and also to $T_e\co \CMGpid \to \CMid$. 
The latter means that $T_e$ sends $(\overline{H},H,\Upsilon, id)$ to $(\overline{\h},\h,\Upsilon_{\ast,e_G}, id)$, which makes notations consistent;  
to see this it is enough to verify that \eqref{eq: prod H bar}, with $\Phi=id$, gives rise to the Lie bracket of \eqref{eq: h bar } by differentiation.   
Moreover, $T_e$ commutes with the projections: 
\begin{prop}
	$R \circ T_e = T_e \circ P$.  
\end{prop}

The previous constructions and remarks can be adapted almost verbatim to the case of \emph{local Lie groups}. Instead to recall the formal definition of this structure, we simply remind that a local Lie group is a smooth manifold $M$ with a distinguished point $e$ and two operations $\mu$ and $\iota$ only partially defined, i.e. defined on a suitable neighborhood of $e$ and satisfying the following compatibility conditions $(1)$ $\mu(e,x)=x=\mu(x,e)$, $(2)$ $\mu(x,\iota(x))=e=\mu(\iota(x),x)$ and $(3)$ $\mu(\mu(x,y),z)=\mu(x,\mu(y,z))$, for all $x,y,z\in M$ \emph{sufficiently} close to $e\in M$. 
To every local Lie group can be associated a Lie algebra whose underlying vector space is the tangent space at $e$ and whose Lie bracket is defined restricting the canonical Lie bracket of $\mathfrak X(M)$ to the (say) left invariant vector fields.
It is worth to observe that every Lie group $G$ is a local Lie group and that every neighborhood $U$ of the identity of a Lie group is a local Lie group, restricting to $U$ both multiplication and inversion map defined on $G$. Another class of local Lie groups is obtained looking at suitable neighborhoods of the $0$ element in a finite dimensional Lie algebra $\g$. In this case the multiplication map is provided by the Baker-Campbell-Hausdorff series, i.e. $\mu(x,y)=\BCH_\g(x,y)$ for all $x,y$, $e$ is the $0$ element and $\iota(x)=-x$ for all $x$. A neighborhood of $0$ is \emph{suitable} if on it the $\BCH$ series is convergent.
This class of examples of local Lie groups will the only one we will consider in this letter. In particular, any local Lie group defined by the Lie algebra $\g$ and its $\BCH$ series will be called a \emph{$\BCH$-group} and it will be denoted simply by $\mathcal G$. In this case $\g$ will be called the Lie algebra underlying $\mathcal G$. In spite of the appearances, our choice to consider only $\BCH$-groups is not really a severe restriction. In fact one can show that every local Lie group, if seen in coordinates, is a $\BCH$-group, see \cite{tao}.
The categories introduced in the first part of this section can be defined trading Lie with local Lie groups. More precisely one can define $\cat{CMGp^{loc}}$, $\cat{CMGp^{loc}_{inv}}$ and, respectively, $\cat{CMG^{loc}_{id}}$. All the comments made and properties discussed about the categories $\cat{CMGp}$, $\cat{CMGp_{inv}}$ and, respectively, $\cat{CMG_{id}}$ can be re-proposed for their local versions.

\section{Universal enveloping algebras and the post-Lie Magnus expansion} \label{sec: Univ. env. alg. and pLMe}

Recall that to a post-Lie algebra $(\h,\plprod)$ one may associate two universal enveloping algebras: that of the Lie algebra $\overline{\h}$ and that of the underlying Lie algebra $\h$. The latter comes equipped with the \emph{Grossman-Larson product} $\ast$, which emerges from the post-Lie structure, making it a bialgebra. Both bialgebras are related by an isomorphism $\Theta\co \mathcal{U}(\overline{\h}) \to (\mathcal{U}(\h),\ast)$ which turns out to be responsible for the existence of the \emph{pLMe} $\chi\co \ca{H} \to \ca{H}$. 

In this section is defined a functor $\ca{U}\co \CM \to \PB$  that provides the above data $( \mathcal{U}(\overline{\h}), \mathcal{U}({\h}),\Theta)$  when restricted to $\CMid$. 
Then is defined an integration functor which  gives rise to the pLMe.
The following diagram gives an overview of the functors considered in the previous and present sections; the bottom line corresponds to the  above discussion. 

\begin{equation*}
\begin{tikzpicture}
[>=stealth,thick,draw=black!65, arrow/.style={->,shorten >=1pt}, point/.style={coordinate}, pointille/.style={draw=red, top color=white, bottom color=red},scale=0.7, photon/.style={decorate,decoration={snake,post length=1mm}}]
\matrix[row sep=8mm,column sep=8mm,ampersand replacement=\&]
{
	\node (-10) {$\cat{CMGp}$};\& \node (-11){$\CM$} 		;\& \node (-12){$\PB$} 							;\& \node (-13){} ;\\
	\node (00) {$\cat{CMGp_{inv}}$};\& \node (01){$\CMinv$} 		;\& \node (02){$\IsoPB$} 				;\& \node (03){$\locGp$} ;\\
	\node (10) {$\cat{CMGp_{id}}$}; \& \node (11){$\CMid$} 			;\& \node (12){}					 	;\& \node (13){} ;\\
	\node (20) {$\cat{Gp}$};  \& \node (21){$\cat{PostLie}$}  ;\& \node (22){$\cat{Bialg}$} 		;\& \node (23){$\cat{Gp^{loc}}$} ;\\
}; 
\path
(-10)     edge[above,->]      		node {$T_e$} 	 		(-11)
(-11)     edge[above,->]      		node (Uu) {$\ca{U}$}  	(-12)
(00)     edge[above,->]      		node {$T_e$}  				(01)
(01)     edge[above,->]     		node (Uu) {$\ca{U}$} 	 	(02)
(02)     edge[above,->,photon]      node {$\text{Int}$}  				(03)
(10)     edge[above,->]     		node {$T_e$} 		 		(11)
(11)     edge[below,->]     		node  {$\ca{U}_{|\iota}$}  	(02)
(20)     edge[above,->]     		node  {$T_e$}  						(21)
(21)     edge[above,->,out=45,in=135]     node (*U)   {$\ast\circ \ca{U}$}  	(22)
(21)     edge[below,->,out=-45,in=-135]   node (Uo[]) {$\ca{U}\circ \llbracket-,-\rrbracket$}  	(22)
(22)     edge[above,->,photon]      	  node (bch1) {$\text{Int}$}  					(23)
(00)     edge[left,left hook->]      node {}  	(-10)
(01)     edge[left,left hook->]      node {}  	(-11)
(02)     edge[left,left hook->]      node {}  	(-12)
(10)     edge[left,left hook->,out=125,in=-125]      node {$s$}  	(00)
(00)     edge[left,->]  						     node {}  		(10)
(01)     edge[left,->]  						     node {$R$}  	(11)
(11)     edge[left,left hook->,out=135,in=-135]      node {$\iota$}  	(01)
(11)     edge[left,->]      node {$\cong$}  	(21)
(Uo[])   edge[right,double,shorten <=8pt,shorten >=8pt,-implies]      		node {$\Theta$}  		(*U)
(11)     edge[left,double,shorten <=8pt,shorten >=8pt,-implies] 		    node {$\Psi$}  		(Uu)
; 
\end{tikzpicture}
\end{equation*}

\begin{notation}
	The universal enveloping algebra of a post-Lie algebra $(\h,\plprod)$ is the universal enveloping algebra of the underlying Lie algebra $\h$. 
	It is a bialgebra when endowed with the shuffle coproduct $\Delta_{sh}\co \ca{U}(\h) \to \ca{U}(\h)^{\ot 2}$. 
	It may be useful to consider Sweedler's notation (without sum): $\Delta_{sh}(X)= X_{(1)}\ot X_{(2)}$ for all $X\in \ca{U}(\h)$. 
\end{notation}

\begin{defn}
The category $\PB$ is as follows. 
The objects are tuples $(A,B,\theta)$ where $A$ and $B$ are bialgebras, $B$ is an $A$--module and $\theta\co A \to B$ is a morphism of $A$--modules and coalgebras. 
The morphisms  are pairs $(f,g)\in \Hom_{bialg}(A,A')\times \Hom_{Modcoalg}(B,B')$ such that $g\circ \theta = \theta' \circ f$.  
That $g$ belongs to $\Hom_{Modcoalg}(B,B')$ means that it is a morphism of coalgebras and that $g(a\cdot b)=f(a)\cdot g(b)$ for all $a\in A$ and $b\in B$. 

Let $\IsoPB$ be the subcategory of $\PB$ of those tuples $(A,B,\theta)$ such that $\theta$ is an isomorphism. 	
\end{defn}

\begin{rem}\label{rmk: PBinv bialg}
	If $(A,B,\theta)$ belongs to $\IsoPB$ then $B$ has another bialgebra structure, given by $b\ast b' := \theta^{-1}(b)\cdot b$ for all $b,b'\in B$. 
	Moreover, $\theta\co A \to (B,\ast)$ is an isomorphism of bialgebras.  
\end{rem}

Let 
\begin{equation*}
\ca{U}\co \CM \to \PB 
\end{equation*}
be the functor that associates to each tuple $(\g,\h,\upsilon,\phi)$ the following tuple $(\ca{U}(\g), (\ca{U}(\h),M),\Theta)$. 
The following construction of the action $M\co \ca{U}(\g) \to \End_{\mathbb K} (\ca{U}(\h))$ and the morphism $\Theta$ are a straightforward generalization of \cite[Section 5]{MQS}; the main steps are given here. 
Since $\upsilon$ has values in $\Der_{\text{Lie}}(\h)$ it can be extended to be with values in the derivations for the algebra $\ca{U}(\h)$. 
By keeping the same notation for this extension, this means that $\upsilon_x(XY)=X\upsilon_x(Y) + \upsilon_x(X)Y$ for each $x\in \h$ and $X,Y\in \ca{U}(\h)$.   
Let $\sigma^\phi\co\mathfrak g\rightarrow\operatorname{End}_{\mathbb K}(\mathcal U(\mathfrak h))$ be the linear application defined by 
\begin{equation*}
\sigma^\phi(x)(X)=\phi(x)\cdot X  \text{ for all }  x\in\g \text{ and } X\in\mathcal U(\mathfrak h),
\end{equation*} 
 and let $M_{(\upsilon,\phi)}:\mathfrak g\rightarrow\operatorname{End}_\mathbb K(\mathcal U(\mathfrak h))$ be the linear map defined 
\begin{equation}
M_{(\upsilon,\phi)}(x)=\upsilon_x+\sigma^\phi_x, \text{ for all } x\in\mathfrak g.\label{eq:eqM}
\end{equation}
The following lemma shows that $M_{(\upsilon,\phi)}$ extends to a morphism of associative algebras $M_{(\upsilon,\phi)}\co \ca{U}(\g) \to \End_{\mathbb K} (\ca{U}(\h))$, providing the action map. 
\begin{lem}
	For all $x,y\in\mathfrak g$, one has 
	\begin{equation}
	M_{(\upsilon,\phi)}([x,y]_{\g})=[M_{(\upsilon,\phi)}(x),M_{(\upsilon,\phi)}(y)]. \label{eq:comm}
	\end{equation}
	In other words, $\mathcal U(\mathfrak h)$ carries a structure of a $(\mathfrak g,[-,-]_\mathfrak g)$--module defined by $M_{(\upsilon,\phi)}\co \g\rightarrow \End_\mathbb K(\ca{U}(\h))$. 
\end{lem}
\begin{proof}
	For every $x,y\in\mathfrak g$ and $a\in\mathfrak h$, it suffices to compare $M^\upsilon_\phi([x,y]_\mathfrak g)(a)$ with 
	$[M^\upsilon_\phi(x),M^\upsilon_\phi(y)](a)$, recalling that $\phi$ satisfies \eqref{eq:crosshom} and 
	$\upsilon\in\operatorname{Hom}_{\text{Lie}}(\mathfrak g,\operatorname{Der}(\mathfrak h))$. 
\end{proof}

The map $\Theta= \Theta_{(\upsilon,\phi)}\co \mathcal U(\mathfrak g)\rightarrow\mathcal U(\mathfrak h)$ is defined on every monomial $X\in\mathcal U(\mathfrak g)$ by
\begin{equation}
\Theta_{(\upsilon,\phi)}(X)=M_{(\upsilon,\phi)}(X)(1)\label{eq:psi}
\end{equation}
and is  extended to all $\mathcal U(\mathfrak g)$ by linearity. 
It is a morphism of coalgebras as well as of left $\mathcal U(\mathfrak g)$--modules; see \cite{KLM} and also \cite[Proposition 28]{MQS}. 
To see that $\ca{U}$ is indeed a functor, it remains to show the following. 
\begin{lem}
Let $(f,g)\co (\g,\h,\upsilon,\phi)\to (\g',\h',\upsilon',\phi')$. 
For all $A\in \ca{U}(\g)$ and $B\in \ca{U}(\h)$, one has     	
\begin{equation*}
\ca{U}(g)(M_{A}(B)) =M'_{\ca{U}(f)(A)}(\ca{U}(g)(B)). 
\end{equation*}
In particular, one has $\ca{U}(g) \circ \Theta =  \Theta' \circ \ca{U}(f)$. 
\end{lem} 
\begin{proof}
 By linearity, it is enough to show the result for $A=a_1\cdots a_m\in \ca{U}(\g)$ and $B=b_1\cdots b_n\in \ca{U}(\h)$ being two monomials.  
 The proof is by induction on $m$. 
 For $m=1$, one has 
\begin{align*}
\ca{U}(g)(M_{a}(B))  
 &= \ca{U}(g) (\upsilon_{a}(B)+ \phi(a)B) \\
 &= \ca{U}(g) (\sum_{1\leq i \leq n} b_1\cdots b_{i-1} \upsilon_{a}(b_i)b_{i+1}\cdots b_n + \phi(a)B) \\
 &= \sum_{1\leq i \leq n} g(b_1) \cdots g(b_{i-1}) \upsilon'_{f(a)}(g(b_i))g(b_{i+1})\cdots g(b_n) + \phi'(f(a))\ca{U}(g)(B) \\
 &= \upsilon'_{f(a)}(\ca{U}(g)(B))  + \phi'(f(a))\ca{U}(g)(B) \\ 
 &= M'_{\ca{U}(f)(a)}(\ca{U}(g)(B)). 
\end{align*}
Let $m\geq 2$. 
Remark that  $\upsilon_{a_1}(M_{a_2}(\cdots (M_{a_m}(B))\cdots )$ can be written as a sum of terms of the form $C_1\upsilon_{a}(C_2)C_3$ where each $C_i\in \ca{U}(\g)$ are monomial of the following form. By writing $C_i$ as $c_{k_1}\cdots c_{k_i}$, the term $c_r$ is of the form 
$\phi(a_s)$, or $\upsilon_{a_{j_1}}( \upsilon_{a_{j_2}} (\cdots \upsilon_{a_{j_s}}(\phi(a_{j_{s+1}}))\cdots ))$ or  
$\upsilon_{a_{j_1}}( \upsilon_{a_{j_2}} (\cdots \upsilon_{a_{j_s}}(B)\cdots ))$ for some indices $\{j_1,...,j_{s+1}\} \subset \{1,...,m\}$. 
Consequently, one has 
\begin{equation*}
\ca{U}(g) \big( \upsilon_{a_1}(M_{a_2}(\cdots (M_{a_m}(B))\cdots ) \big) = \upsilon'_{f(a_1)}(M'_{f(a_2)}(\cdots (M'_{f(a_m)}(\ca{U}(g)(B))\cdots ). 
\end{equation*}
Therefore, one has 
\begin{multline*}
\ca{U}(g) (M_{a_1\cdots a_m}(B)) 
=\ca{U}(g) \Big(\upsilon_{a_1}(M_{a_2\cdots a_m}(B)) + \phi(a_1)M_{a_2...a_m}(B)\Big)
=  M'_{\ca{U}(f)(a_1\cdots a_m)}(\ca{U}(g)(B)). 
\end{multline*} 
\end{proof}

If $\phi$ is invertible, then so is $\Theta$; see \cite[Theorem 29]{MQS}. 
Therefore, the functor $\ca{U}$ restricts to a functor 
\begin{equation*}
\ca{U}\co \CMinv \to \IsoPB. 
\end{equation*}

Since $(\iota,R)$ is an adjoint equivalence, the counit provides a natural isomorphism  $\Psi=\ca{U}\epsilon \co R\ca{U}_{|\iota} \to  \ca{U}$. 
In particular one has 
\begin{equation}\label{eq: Theta phi}
\Theta_{(\upsilon\circ \phi^{-1},id)} = \Theta_{(\upsilon,\phi)} \circ \ca{U}(\phi^{-1}).
\end{equation}

\begin{rem}\label{rmk: GL product and D-alg}
By Remark \ref{rmk: PBinv bialg}, the morphism $\Theta_{(\upsilon\circ \phi^{-1},id)}$ is a morphism of bialgebras, so one recovers the initial viewpoint of  \cite{KLM}, see also \cite{GO1,GO2}.  
In particular, the resulting $\ast$ product on $\ca{U}(\h)$ is the \emph{Grossman-Larson} product; it can be constructed as follows.

The post-Lie product on $\h$, defined  via \eqref{eq: postlie from hh,v ,id}, can be extended to a map 
$\plprod\co \ca{U}(\h)^{\ot 2} \to \ca{U}(\h)$ with the following properties. 
For all $X,Y$ and $Z$ in $\mathcal{U}(\g)$ and $x$ and $y$ in $\g$, one has: 
\begin{enumD}
	\item\label{D bial item1} $1\triangleright X=X$ and $X\triangleright 1=0$; 
	\item\label{D bial item4} $X\triangleright(Y\cdot Z)=(X_{(1)}\triangleright Y)\cdot(X_{(2)}\triangleright Z)$; and,  
	\item\label{D bial item5} $(x\cdot X)\triangleright y=x\triangleright (X\triangleright y)-(x\triangleright X)\triangleright y$. 
\end{enumD}
The resulting structure is known as a \emph{$D$--bialgebra} structure $(\mathcal{U}(\h),\Delta_{sh},\plprod)$; see \cite{MQS}. 
The Grossman-Larson product is given by 
\begin{equation}
\begin{split}
\ast \co \mathcal{U}(\h)^{\ot 2} &\to \mathcal{U}(\h) \\
X\ot Y & \mapsto  X_{(1)}(X_{(2)}\plprod Y). 
\end{split}
\end{equation}
\end{rem}

\subsection{Integration of post-Lie algebras} 

Let $\CMinvfin$ denotes the subcategory of  $\CMinv$ of those tuples where the Lie algebras are \emph{finite dimensional}  
and let $\PBcomp$ be the obtained from $\IsoPB$ by requiring the bialgebras to be \emph{complete}. 
The functor  $\ca{U}\co \CMinvfin \to \IsoPB$ induces, after completion, a functor 
$\widehat{\ca{U}}\co \CMinvfin \to \PBcomp$. 
Its image forms a category $\ImU$: the objects of $\ImU$ are tuples of the form $(\widehat{\ca{U}}(\g),(\widehat{\ca{U}}(\h),M),\Theta)= \widehat{\ca{U}}(\g,\h,\upsilon,\phi)$, and morphisms are of the form $\widehat{\ca{U}}(f,g)$ for morphisms $(f,g)$ in $\CMinvfin$.  
In what follows is defined an integration functor 
\begin{equation*}
\text{Int} \co \ImU \to  \locGp.
\end{equation*}
To any tuple $(\widehat{\ca{U}}(\g),(\widehat{\ca{U}}(\h),M),\Theta)$ in $\ImU$ one may associate the following tuple 
$(\ca{G}, \ca{H},\varUpsilon, \varPhi)$.

The map 
$\varUpsilon\co \ca{G} \to \Aut(\ca{H})$ 
is given by 
\begin{equation*}
\varUpsilon= \Exp(\upsilon)
\end{equation*}
 where $\upsilon_x(y):= M_x(y) - \Theta(x)\cdot y$ for all $x\in \g$ and $y\in \h$.  
\begin{lem}
	$\varUpsilon$ is a morphism of local groups. 
\end{lem}
\begin{proof}
Note that if $d\in\Der(\g)$, then $\Exp(d)\in\Aut(\g)\subset\Aut(\ca{G})$. 
Furthermore, if $\g$ and $\h$ are two (finite dimensional) Lie algebras and $\upsilon\in\Hom_{\text{Lie}}(\g,\Der(\h))$, then one has 
\begin{equation*}
\operatorname{Exp}(\upsilon_x)\operatorname{Exp}(\upsilon_y)=\operatorname{Exp}(\operatorname{BCH}_{\operatorname{End}(\mathfrak h)}(\upsilon_x,\upsilon_y))=\operatorname{Exp}(\upsilon_{\operatorname{BCH}_{\mathfrak g}(x,y)}),\,\forall x,y\in\mathfrak g,
\end{equation*}
which gives the result. 
\end{proof}
The map 
$\varPhi\co \ca{G} \to \ca{H}$ 
is defined as 
\begin{equation*}
\varPhi = \varPhi_{(\upsilon,\phi)} = \log_{\h} \circ \Theta_{(\upsilon,\phi)} \circ \exp_{\g}.
\end{equation*}
\begin{lem}
	 $\varPhi$ is a crossed morphism of local groups. 
\end{lem}

\begin{proof}
It is a direct consequence of Theorem \ref{thm:intpostLie} stated hereafter. 
Indeed formula \eqref{eq:bch1} can be written as 
\begin{equation*}\label{eq:bch2}
\varPhi(\BCH_{\g}(x,y))=\BCH_{\h}(\varPhi(x),\Exp (\upsilon_{x}) \varPhi(y)). 
\end{equation*}
\end{proof}

The integration functor $\text{Int}$ is given by 
$\text{Int} (\widehat{\ca{U}}(\g),(\widehat{\ca{U}}(\h),M),\Theta) = (\ca{G}, \ca{H},\varUpsilon, \varPhi)$. 
A direct verification shows that it is indeed a functor.

\subsection{The post-Lie Magnus expansion}

Observe that since $\phi$ is invertible, so is $\varPhi_{(\upsilon, \phi)}$. Its inverse $\chi_{(\upsilon, \phi)} \co  \ca{H} \rightarrow \ca{G}$ is therefore given by 
\begin{equation}\label{eq:chi}
\chi_{(\upsilon , \phi )} = \log_{\g}\circ (\Theta_{(\upsilon, \phi)})^{-1} \circ \exp_{\h}. 
\end{equation}
\begin{defn}
	The map $\chi_{(\upsilon, \phi)}$ is called the \emph{post-Lie Magnus expansion associated to} $(\g,\h,\upsilon,\phi) \in \CMinv$. 
\end{defn}

In analogy to \cite[Proposition 39]{MQS}, one can prove the following result.

\begin{thm}\label{thm:intpostLie}
For all $a,b\in \h$ one has 
	\begin{equation}\label{eq:bch1}
	\operatorname{BCH}_{\g}(\chi_{(\upsilon, \phi)}(a), \chi_{(\upsilon, \phi)}(b))
	=\chi_{(\upsilon, \phi)}\big(\operatorname{BCH}_{\h}
		\big(a,\operatorname{Exp}(\upsilon_{\chi_{(\upsilon, \phi)}(a)})b\big)\big).  
	\end{equation}
\end{thm} 
The proof of this result is based on the following two preliminary lemmas. 
First recall that, by Remark \ref{rmk: PBinv bialg}, the bialgebra $\ca{U}(\h)$ can be endowed with another product $\ast$.  
Also recall that, by \eqref{eq: postlie from hh,v ,id}, the map $\plprod\co a\ot b\mapsto \upsilon_{\phi^{-1}(a)}(b)$ defines a post-Lie product on $\h$.

Let $\sharp \co \h\times \h\rightarrow \h$ be defined by
\begin{equation*}\label{eq:compositio1}
a\sharp b=\log_\h(\exp_\h(a)\ast\exp_\h(b)) \text{ for all } a,b\in \h. 
\end{equation*}
\begin{lem} \label{eq:composition2}
For all $a,b\in \h$ one has 
	\begin{equation*}
	a\sharp b=\operatorname{BCH}_\h(a,\exp_\h(a) \triangleright b). 
	\end{equation*}
\end{lem}
\begin{proof}
	This result was proven in \cite{F-MK} and that proof extends without modification to this context.  
\end{proof}
\begin{lem}
	For all $a,b\in\mathfrak h$ one has 
	\begin{equation}
	\exp_\cdot(a)\triangleright b
	=\operatorname{Exp}(\upsilon_{\chi_{(\upsilon, \phi)}(a)})b,\label{eq:identity}
	\end{equation}
	where the right hand side of the previous formula reads as
	\[
	b+\upsilon_{\chi_{(\upsilon, \phi)}(a)}(b)+\frac{1}{2}\upsilon_{\chi_{(\upsilon, \phi)}(a)}\big(\upsilon_{\chi_{(\upsilon, \phi)}(a)}(b)\big)+\frac{1}{3!}\upsilon_{\chi_{(\upsilon, \phi)}(a)}\big(\upsilon_{\chi_{(\upsilon, \phi)}(a)}(\upsilon_{\chi_{(\upsilon, \phi)}(a)}(b))\big)+\cdots
	\]
\end{lem}
\begin{proof}
	Recall that the extension of the post-Lie product to $\mathcal U(\mathfrak h)$ endowed the latter with a structure of a $D$-bialgebra, see Definition 19 in \cite{MQS}. In this case one has 
	\begin{enumerate}
		\item[(i)] $a\ast A=a\cdot A+a\plprod A$, for all $a\in \h$, see Formula $(4.20)$ pag. 570 in \cite{MQS}; 
		\item[(ii)] ($a\cdot A)\plprod a'=a\plprod (A\plprod a')-(a\plprod A)\plprod a'$, see Formula $(\rm{D}.5)$ pag. 566 in \cite{MQS},
	\end{enumerate}
	for all $a,a'\in \h$ and $A\in\h(V)$. Plugging $(ii)$ into $(i)$ one obtains
	\[
	(a\ast A)\plprod a'=a\plprod (a'\plprod A).
	\]
	The proof of the statement now can be obtained using a simple induction on the length of the monomials in the RHS of \eqref{eq:identity}, applying $(i)$ and $(ii)$ above recalled. 
\end{proof}

After these observations the proof of the Theorem \ref{thm:intpostLie} is formally identical to the one presented in \cite{MQS} and for this reason it is not presented again. 
\\

The functoriality of the construction of the pLMe implies some relations between the different pLMe's that rely on relations between the source objects. 
More precisely, one has the following  relations.  

Let $X=(\g,\h,\upsilon,\phi)$ be an object of $\CMinvfin$. 
Recall the equation \eqref{eq: Theta phi} which, by applying $\text{Int}$, gives 
$\varPhi_{(\upsilon\circ \phi^{-1},id)} = \varPhi_{(\upsilon,\phi)} \circ \text{Int} (\widehat{\ca{U}}(\phi^{-1}))$. 
In other words, one has  
\begin{equation*}
\chi_{(\upsilon,\phi)} =  \phi^{-1} \circ \chi_{(\upsilon\circ \phi^{-1},id)}. 
\end{equation*}

Let $(f,g)\co (\g,\h,\upsilon,\phi) \to (\g',\h',\upsilon',\phi')$ be a morphism in $\CMinvfin$. 
By applying $\text{Int} \circ \widehat{\ca{U}}\circ R$, one obtains 
\begin{equation*}
\varPhi_{(\upsilon \circ \phi^{-1},id)} \circ \text{Int}(\widehat{\ca{U}}(g) )
= \text{Int}( \widehat{\ca{U}}(g)) \circ \varPhi_{(\upsilon' \circ (\phi')^{-1},id)}.  
\end{equation*}
In particular, if $(f,g)$ is an isomorphism one has  
\begin{equation*}
\chi_{(\upsilon' \circ (\phi')^{-1},id)} 
= g^{-1} \circ  \chi_{(\upsilon \circ \phi^{-1},id)} \circ   g.
\end{equation*}

\section{Computing the Post-Lie Magnus expansion}\label{sec: Computing pLMe}

In this section two combinatorial interpretations of the coefficients associated to any forest of the pLMe are given. 
Both interpretations are based on a notion of nested tubings. 
The first method is concerned with \emph{vertical} nested tubings and allows to compute the coefficients associated to any forest recursively. 
The second method is concerned with \emph{horizontal} nested tubings and allows to express these coefficients in a closed form. 
\\

\subsection{Planar trees and forests}\label{sec: Operad of PRT}

\begin{defn} 
	A \emph{planar rooted tree} is an isomorphism class of contractible graphs, embedded in the plane, and endowed with a distinguished vertex, called the \emph{root}, to which is attached an adjacent half-edge, called the \emph{root-edge} of the planar tree. 
\end{defn}

For a planar rooted tree $T$, we let $V(T)$ be the set of all its vertices. On it, we consider two orders: 
\begin{itemize}
	\item  The \emph{level partial} order $\prec$ defined by orienting the edges of $T$ towards the root, except the root-edge. 
	For two  vertices $u$ and $v$ of $ V(T)$, we write $v \prec u$ if there is a string of oriented edges from $v$ to $u$. In particular, the root is  maximal for this partial order. 
	\item The \emph{canonical linear} order $<$: starting from the root-edge of $T$, we run along $T$ in the clockwise direction, passing trough each edge once per direction. The order we meet the vertices for the first time gives the order $<$. In particular the root is the minimal element for $<$. 
\end{itemize}

Pictorially, our trees are drawn with the root at the bottom, and the order on the set of the incoming edges of a vertex is given by the clockwise direction, i.e. from the left to the right.

From now on, when there is no ambiguity, planar rooted trees are simply called trees.  

\begin{example}
	For any two trees $R$ and $S$, let $C(\bullet; R,S)$ be the corolla with an unlabeled vertex $v$ of arity $2$ as root; the roots of $R$ and $S$ are input edges of $v$ in this order. 
\end{example}

\begin{defn}
	Let $T$ be a tree and $v$ a vertex of it. Consider a small disc centered in $v$. The outgoing and incoming edges of $v$ cut the disc into connected components. If $v$ has at least one incoming edge, the \emph{left side of $v$} is the connected components delimited by the outgoing edge and the first incoming edge of $v$. Otherwise, its \emph{left side} is the unique connected component of the cut disc.    
\end{defn}

\begin{example}
A vertex $v$ and its left side (the darkest gray region):
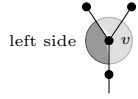
\begin{equation*} 
	\pgfdeclarelayer{background}
	\pgfsetlayers{background,main} 
	\begin{tikzpicture}
	[  
	level distance=0.45cm, 
	level 2/.style={sibling distance=0.6cm}, sibling distance=0.6cm]
	\node (1) [my circle] {} [grow=up]
	{
		child {node (3) [my circle,label={[label distance=-.03cm]0:\tiny{$v$}},label={[label distance=.25cm]180:\tiny{left side}}] {}
			child {node (5) [my circle] {}}
			child {node (6) [my circle] {}} }
	};
	\draw [-] (0,-.25) -- (0,0) ;
	\begin{pgfonlayer}{background}
	\tkzMarkAngle[fill=gray,size=0.31cm,opacity=.8](6,3,1)
	\tkzMarkAngle[fill=gray,size=0.31cm,opacity=.25,draw opacity=.3](5,3,6)
	\tkzMarkAngle[fill=gray,size=0.31cm,opacity=.25,draw opacity=.3](1,3,5)
	\end{pgfonlayer}
	\end{tikzpicture}
	\end{equation*}
\end{example}

\begin{defn}
	A \emph{forest} is a (non commutative) word of trees. 
	For $n\geq 1$, the forest of $n$ times the tree with one vertex is denoted by $\bullet^{\times n}$ and is called \emph{horizontal}.  
\end{defn}

Trees and forests can be grafted at vertices, as follows.

\begin{notation}\label{notation: operations graft}
	\begin{enumerate}
		\item For any two trees $R$ and $T$ we let $R \triangleright_{v} T$ be the tree obtained by grafting the root-edge of $R$ at the vertex $v$, on its left side.  
		\item\label{item: operations grafting/concat} Let $n\geq 1$  and  $n_0+ n_1+...+n_k=n$ be a partition of $n$ such that $n_i\geq 1$ for $1\leq i \leq k$ and $n_{0}\geq 0$. 
		Let $F$ be a forest and let  $v_1,...,v_k$ be $k$ vertices of $F$. 
		For any forest $E$ of $n$ trees, let 
		$E \ltimes_{v_1,...,v_k}^{n_0, n_1,...,n_k} F$ be the forest obtained from $F$ by grafting the first $n_1$ roots of $E$ to the left-side of $v_1$, the next $n_2$ roots of $\bullet^{\times n}$ to the left-side of $v_2$, and so on until $n_k$; the  $n_{0}$ last trees are concatenated to the left of the so-obtained forest.    
		In particular,
		\begin{itemize}
			\item for $k=0$, the operation $\ltimes_{\emptyset}^{n}$ is the concatenation operation that we simply denote by $\times$; 
 			\item for $k=1$ and $n_{0}=0$, the operation $\ltimes_{v}^{0,n}$ is the grafting of all the roots to a single vertex $v$ that we simply denote by $\plprod_v$;  
			\item for $n_{0} = 0$, we write $ \ltimes_{v_1,...,v_k}^{0,n_1,...,n_k} $ as $\plprod_{v_1,...,v_k}^{n_1,...,n_k}$. 
		\end{itemize}
	\end{enumerate}
\end{notation}

For instance, one has 
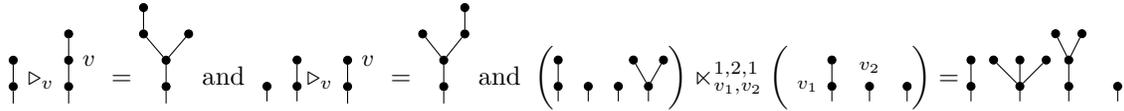
\begin{equation*}
\begin{tikzpicture}
[  level distance=0.4cm, level 2/.style={sibling distance=0.6cm}, sibling distance=0.6cm,baseline=2.5ex,level 1/.style={level distance=0.35cm},baseline=2.5ex]
\node [] {} [grow=up]
{	
	child {node [my circle]  {}
		child {node [my circle]  {} 
		}	
	}
};
\end{tikzpicture}   
\triangleright_{v} 
\begin{tikzpicture}
[  level distance=0.4cm, level 2/.style={sibling distance=0.6cm}, sibling distance=0.6cm,baseline=2.5ex,level 1/.style={level distance=0.35cm},baseline=2.5ex]
\node [] {} [grow=up]
{	
	child {node [my circle]  {}
		child {node [my circle,label=right:\small{$v$}]  {} 
			child {node [my circle]  {} 
			}
		}
	}
}
;
\end{tikzpicture}
=
\begin{tikzpicture}
[  level distance=0.4cm, level 2/.style={sibling distance=0.6cm}, sibling distance=0.6cm,baseline=2.5ex,level 1/.style={level distance=0.35cm},baseline=2.5ex]
\node [] {} [grow=up]
{	
	child {node [my circle]  {}
		child {node [my circle]  {} 
			child {node [my circle]  {} 	
			}
			child {node [my circle]  {} 	child {node [my circle]  {} }} 
		}	
	}
};
\end{tikzpicture}
\text{ and }~
\begin{tikzpicture}
[level 1/.style={level distance=0cm,sibling distance=.4cm }, level 2/.style={level distance=0.2cm,sibling distance=.6cm }, level 3/.style={sibling distance=0.6cm,level distance=0.35cm}, sibling distance=0.6cm,baseline=1.5ex]
\node [] {} [grow=up]	 
{	child  { edge from parent[draw=none]  child {node [my circle] (A)  {} 	child {node [my circle]  (B) {} 		} 		}}
	child { edge from parent[draw=none] child  {	 {node [my circle,black] (l1) {} 
			}
	}}
};
\end{tikzpicture}  
\plprod_v 
\begin{tikzpicture}
[  level distance=0.4cm, level 2/.style={sibling distance=0.6cm}, sibling distance=0.6cm,baseline=2.5ex,level 1/.style={level distance=0.35cm},baseline=2.5ex]
\node [] {} [grow=up]
{	
	child {node [my circle]  {}
		child {node [my circle,label=right:\small{$v$}]  {} 
		}	
	}
};
\end{tikzpicture}   
=
\begin{tikzpicture}
[  level distance=0.4cm, level 2/.style={sibling distance=0.56cm}, sibling distance=0.6cm,baseline=2.5ex,level 1/.style={level distance=0.35cm},baseline=2.5ex]
\node [] {} [grow=up]
{	
	child {node [my circle]  {}
		child {node [my circle]  {} 
			child {node [my circle]  {} 	child {node [my circle]  {} }	
			}
			child {node [my circle]  {} } 
		}	
	}
};
\end{tikzpicture}
\text{ and }
\left(
\begin{tikzpicture}
[level 1/.style={level distance=0cm,sibling distance=.4cm }, level 2/.style={level distance=0.2cm,sibling distance=.6cm }, level 3/.style={sibling distance=0.4cm,level distance=0.35cm}, sibling distance=0.6cm,baseline=1.5ex]
\node [] {} [grow'=up]	 
{	
	child  { edge from parent[draw=none]  
		child {node [my circle] (A)  {} 	
			child {node [my circle]  (B) {} 		} 		}
	}
	child { edge from parent[draw=none] 
		child  {	 {node [my circle] (T2) {}}	}
	}
	child { edge from parent[draw=none] 
		child  {	 {node [my circle] (T3) {}}	}
	}
	child  { edge from parent[draw=none]  
		child {node [my circle] (T4lev0)  {} 	
			child {node [my circle]  (T4lev1) {}}
			child {node [my circle]  (T4lev12) {}}}
	}
};
\end{tikzpicture}  
\right)
\ltimes_{v_1,v_2}^{1,2,1}
\left(
\begin{tikzpicture}
[level 1/.style={level distance=0cm,sibling distance=.5cm }, level 2/.style={level distance=0.2cm,sibling distance=.6cm }, level 3/.style={sibling distance=0.6cm,level distance=0.35cm}, sibling distance=0.6cm,baseline=1.5ex]
\node [] {} [grow'=up]	 
{	
	child  { edge from parent[draw=none]  
		child {node [my circle,label=left:\scriptsize{$v_1$}] (A)  {} 	
			child {node [my circle]  (B) {} 		} 		}
	}
	child { edge from parent[draw=none] 
		child  {	 {node [my circle,label=above:\scriptsize{$v_2$}] (T2) {}}	}
	}
	child { edge from parent[draw=none] 
		child  {	 {node [my circle] (T3) {}}	}
	}
};
\end{tikzpicture}  
\right)
= 
\begin{tikzpicture}
[level 1/.style={level distance=0cm,sibling distance=.65cm }, level 2/.style={level distance=0.2cm,sibling distance=.6cm }, level 3/.style={sibling distance=0.35cm,level distance=0.35cm}, sibling distance=0.6cm,baseline=1.5ex]
\node [] {} [grow'=up]	 
{	
	child  { edge from parent[draw=none]  
		child {node [my circle] (A)  {} 	
			child {node [my circle]  (B) {} 		} 		}
	}
	child { edge from parent[draw=none] 
		child {	{node [my circle] (T1) {}
				child {node [my circle]  (T1lev11) {}	}
				child {node [my circle]  (T1lev12) {}	}
				child {node [my circle]  (T1lev13) {}	}}}
	}
	child { edge from parent[draw=none] 
		child {	{node [my circle] (T2) {}
				child {node [my circle]  (T2lev1) {}	 	
					child {node [my circle]  (T2lev21) {}	}
					child {node [my circle]  (T2lev22) {}	}}}}
	}
	child { edge from parent[draw=none] 
		child {	{node [my circle] (T3) {}}}
	}
};
\end{tikzpicture}.
\end{equation*}
In the second case one has $k=1$, $n_0=0$ and $n_1=2$; in the last case one has $k=2$, $n_0=1$, $n_1=2$ and $n_3=1$. 
\\

We also will be led to consider trees with labelings, or more in general, with partial labelings. 
\begin{defn} 
	Let $T$ be a tree and let $U$ be a subset of  $V(T)$.  
	A \emph{$U$--label} of $T$ is a bijection $l\co U \rightarrow \{1,...,n\}$. 
	A tree $T$ equipped with a $U$--label is called \emph{partially labeled}. 
\end{defn}

\begin{example}
	Examples of partially labeled trees: \\
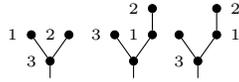
	\begin{equation*}
	\begin{tikzpicture}
	[baseline, my circle/.style={draw, fill, circle, minimum size=3pt, inner sep=0pt},  
	level 2/.style={sibling distance=0.5cm,level distance=0.35cm}, sibling distance=0.5cm, level 1/.style={level distance=0.35cm}]
	\node {} [grow=up]
	{child{node [my circle,label=left:\tiny{$3$}] {} child {node (A) [my circle,label=left:\tiny{$2$}]  {}
			} child {node [my circle,label=left:\tiny{$1$}] {}} }};
	\end{tikzpicture}~~
	\begin{tikzpicture}
	[baseline, my circle/.style={draw, fill, circle, minimum size=3pt, inner sep=0pt},  
	level 2/.style={sibling distance=0.5cm,level distance=0.35cm}, sibling distance=0.5cm, level 1/.style={level distance=0.35cm}]	
	\node {} [grow=up]
	{child{node [my circle] {} child {node (A) [my circle,label=left:\tiny{$1$}] {} child {node [my circle,label=left:\tiny{$2$}] {}}}
			child {node  [my circle,label=left:\tiny{$3$}]  {}
	}  }};
	\end{tikzpicture}~~
	\begin{tikzpicture}
	[baseline, my circle/.style={draw, fill, circle, minimum size=3pt, inner sep=0pt},  
	level 2/.style={sibling distance=0.5cm,level distance=0.35cm}, sibling distance=0.5cm, level 1/.style={level distance=0.35cm}]
	\node {} [grow=up]
	{child{node [my circle,label=left:\tiny{$3$}] {} child {node (A) [my circle, label=right:\tiny{$1$}] {} child {node [my circle,label=right:\tiny{$2$}] {}}}
			child {node  [my circle]  {}
	}  }};
	\end{tikzpicture} 
	\end{equation*}
\end{example}

\subsection{Definition of $\TB$}
In this section is reminded the minimal material about the operad $\PSB$; we refer to \cite{MQS} for completeness. 
\\

For $n\geq 1$, let  $\mathcal L(n)$ be the $\mathbb{K}$--vector space generated by the fully labeled trees with $n$ vertices. 
For each $n\geq 2$ let $\mathcal{W} (n)$ be the $\mathbb{K}$--vector space generated by trees $T$ with partial labeling $l\co U\to \{1...,n\}$  that satisfy: 
\begin{enumerate}
	\item[(a)] the root of $T$ is unlabeled; 
	\item[(b)] if a vertex of $T$ is unlabeled, then so is its $\prec$--successor;  
	\item[(c)] each unlabeled vertex of $T$ has exactly two incoming edges. 
\end{enumerate}

Let 
\begin{equation*}
\ca{LW}(1):=\ca{L}(1) \text{ and }  \ca{LW}(n):=\mathcal L(n)\oplus\mathcal W(n) \text{ for } n\geq 2. 
\end{equation*}

In \cite{MQS}, a structure of operad was provided on the collection $\{\ca{LW}(n)\}_n$. One may therefore consider the following ideal  $\mathcal{I}\subset \mathcal{LW}$ generated by 
\[\Biggl \{ \begin{tikzpicture}
[baseline, my circle/.style={draw, fill, circle, minimum size=3pt, inner sep=0pt},  
level 2/.style={sibling distance=0.5cm,level distance=0.4cm}, sibling distance=0.5cm, level 1/.style={level distance=0.4cm}]
\node [my circle]  {} [grow=up]  {child {node (A) [my circle,label=left:\tiny{$2$}]  {}} 
	child {node (B) [my circle,label=left:\tiny{$1$}] {}} };
\draw [-] (0,-.25) -- (0,0) ;

\end{tikzpicture} 
-
\begin{tikzpicture}
[baseline, my circle/.style={draw, fill, circle, minimum size=3pt, inner sep=0pt},  
level 2/.style={sibling distance=0.5cm,level distance=0.4cm}, sibling distance=0.5cm, level 1/.style={level distance=0.4cm}]

\node [my circle]  {} [grow=up] {child {node (A) [my circle,label=left:\tiny{$1$}]  {}
	} child {node  [my circle,label=left:\tiny{$2$}] {}} };
\draw [-] (0,-.25) -- (0,0) ;
\end{tikzpicture} ,
\begin{tikzpicture}
[baseline, my circle/.style={draw, fill, circle, minimum size=3pt, inner sep=0pt},  
level 2/.style={sibling distance=0.5cm,level distance=0.4cm}, sibling distance=0.5cm, level 1/.style={level distance=0.4cm}]

\node [my circle] {} [grow=up] { child {node (A)  [my circle] {}  child {node  [my circle,label=left:\tiny{$3$}]  {}}
		child {node  [my circle,label=left:\tiny{$2$}]  {}
	} } child {node  [my circle,label=left:\tiny{$1$}]  {}
} };
\draw [-] (0,-.25) -- (0,0) ;
\end{tikzpicture} 
-
\begin{tikzpicture}
[baseline, my circle/.style={draw, fill, circle, minimum size=3pt, inner sep=0pt},  
level 2/.style={sibling distance=0.5cm,level distance=0.4cm}, sibling distance=0.5cm, level 1/.style={level distance=0.4cm}]

\node [my circle] {} [grow=up] {child {node (A)  [my circle,label=right:\tiny{$3$}]  {}}
	child {node   [my circle] {}  child {node  [my circle,label=left:\tiny{$2$}]  {}}
		child {node  [my circle,label=left:\tiny{$1$}]  {}
} }  };
\draw [-] (0,-.25) -- (0,0) ;
\end{tikzpicture} 
-
\begin{tikzpicture}
[baseline, my circle/.style={draw, fill, circle, minimum size=3pt, inner sep=0pt},  
level 2/.style={sibling distance=0.5cm,level distance=0.4cm}, sibling distance=0.5cm, level 1/.style={level distance=0.4cm}]

\node [my circle] {}  [grow=up] {child {node (A)  [my circle] {}  child {node  [my circle,label=left:\tiny{$3$}]  {}}
		child {node  [my circle,label=left:\tiny{$1$}]  {}
	} } child {node  [my circle,label=left:\tiny{$2$}]  {}
} };
\draw [-] (0,-.25) -- (0,0) ;
\end{tikzpicture} 
\Biggr \}. \]

For each $n\geq1,$ we let 
$\TB(n):= \mathcal{LW}(n)/ \mathcal{I}(n)$. 
\begin{thm}\label{th: iso}{\cite{MQS}}
	The collection $\{\TB(n)\}_n$ is endowed with a structure of symmetric operad which makes it isomorphic to the operad $\PostLie$. 
\end{thm}

Let us make this operadic structure explicit for any two trees $T\in \PSB(m)$ and $R\in \PSB(n)$  that are fully labeled. 
Let $v$ be the vertex of $T$ that is labeled by $i$; let $k$ be the number of its incoming edges. 
For a map $\phi\co \{1,...,k\} \to V(R)$, let $T\circ_i^{\phi} R$ to be the tree obtained by 
substituting the vertex labeled by $i$ by the tree $R$, and then grafting the incoming edges  of $i$ to the labeled vertices of $R$ following the map $\phi$. 
The grafting is required to be performed in such a way that it respects the natural order of each fiber of $\phi$. 
This means that if $\phi(v)^{-1}=\{i_1<i_2<...<i_s\}\subset \{1<...<k\}$, then, in the resulting tree, the incoming edge resulting from the grafting of $i_r$--th incoming edge is the $r$--th incoming edge of $v$. 
The labeling of $T\circ_i^{\phi} R$ is given by classical re-indexation.  

The partial composition of $T$ and $R$ at $i$ is:  
\begin{equation}\label{eq: explicit partial compo partial planar tree}
T\circ_i R = \sum_{\phi} T \circ_i^{\phi} R,
\end{equation} 
where $\phi$ runs through the set of maps from $\{1,...,k\}$ to $V(R)$. 
For instance, one has   
\begin{equation*}
\begin{tikzpicture}
[baseline, my circle/.style={draw, fill, circle, minimum size=3pt, inner sep=0pt}, level distance=0.4cm, 
level 2/.style={sibling distance=0.6cm}, sibling distance=0.6cm,baseline=1.5ex]
\node [my circle,label=left:\tiny{$1$}] {} [grow=up]
{
	child {node [my circle,label=left:\tiny{$3$}]  {}} 
	child {node [my circle,label=left:\tiny{$2$}] {}} 
};
\draw [-] (0,-.25) -- (0,0) ;
\end{tikzpicture}
\circ_1 
\begin{tikzpicture}
[baseline, my circle/.style={draw, fill, circle, minimum size=3pt, inner sep=0pt}, level distance=0.4cm, 
level 2/.style={sibling distance=0.6cm}, sibling distance=0.6cm,baseline=1.5ex]
\node [my circle,label=left:\tiny{$1$}] {} [grow=up]
{
	child {node [my circle,label=left:\tiny{$2$}]  {}} 
};
\draw [-] (0,-.25) -- (0,0) ;
\end{tikzpicture}
=
\begin{tikzpicture}
[baseline, my circle/.style={draw, fill, circle, minimum size=3pt, inner sep=0pt}, level distance=0.4cm, 
level 2/.style={sibling distance=0.6cm}, sibling distance=0.6cm,baseline=1.5ex]
\node [my circle,label=left:\tiny{$1$}] {} [grow=up]
{
	child {node [my circle,label=left:\tiny{$2$}]  {}} 
	child {node [my circle,label=left:\tiny{$4$}] {}} 
	child {node [my circle,label=left:\tiny{$3$}] {}}
};
\draw [-] (0,-.25) -- (0,0) ;
\end{tikzpicture}
+
\begin{tikzpicture}
[baseline, my circle/.style={draw, fill, circle, minimum size=3pt, inner sep=0pt}, level distance=0.4cm, 
level 2/.style={sibling distance=0.6cm}, sibling distance=0.6cm,baseline=1.5ex]
\node [my circle,label=left:\tiny{$1$}] {} [grow=up]
{
	child {node [my circle,label=left:\tiny{$2$}]  {} 
		child {node [my circle,label=left:\tiny{$4$}] {}} 
	}
	child {node [my circle,label=left:\tiny{$3$}] {} }
};
\draw [-] (0,-.25) -- (0,0) ;
\end{tikzpicture}
+
\begin{tikzpicture}
[baseline, my circle/.style={draw, fill, circle, minimum size=3pt, inner sep=0pt}, level distance=0.4cm, 
level 2/.style={sibling distance=0.6cm}, sibling distance=0.6cm,baseline=1.5ex]
\node [my circle,label=left:\tiny{$1$}] {} [grow=up]
{
	child {node [my circle,label=left:\tiny{$2$}]  {} 
		child {node [my circle,label=left:\tiny{$3$}] {}} 
	}
	child {node [my circle,label=left:\tiny{$4$}] {} }
};
\draw [-] (0,-.25) -- (0,0) ;
\end{tikzpicture}
+
\begin{tikzpicture}
[baseline, my circle/.style={draw, fill, circle, minimum size=3pt, inner sep=0pt}, level distance=0.4cm, 
level 2/.style={sibling distance=0.6cm}, sibling distance=0.6cm,baseline=1.5ex]
\node [my circle,label=left:\tiny{$1$}] {} [grow=up]
{
	child {node [my circle,label=left:\tiny{$2$}]  {} 
		child {node [my circle,label=left:\tiny{$4$}] {}} 
		child {node [my circle,label=left:\tiny{$3$}] {}}
	}
};
\draw [-] (0,-.25) -- (0,0) ;
\end{tikzpicture}
.
\end{equation*}

\subsection{The free post-Lie algebra}\label{rem: free postlie}
Given an operad $\mathcal{O}$ and a vector space $V$, we denote by $\mathcal{O}(V)$ the free $\mathcal{O}$--algebra generated by $V$. It is explicitly given by $ \mathcal{O}(V)= \bigoplus_{n\geq 0} \mathcal{O}(n)\ot_{\mathbb{S}_n} V^{\ot n}$. 
By Theorem \ref{th: iso}, we know that $\TB(\K)$ is the free post-Lie algebra on $\K$, which is the vector space generated by trees of $\TB$, with a unique label. In other words, if we let $\K=\K<\bsq>$ for a generator $\bsq$, then $\TB(\K)$ is generated by the set 
\begin{equation*}\label{eq: set gen of free postlie}
\mathcal{G}= \bigg\{
\begin{tikzpicture}
[  level distance=0.5cm, level 2/.style={sibling distance=0.6cm}, sibling distance=0.6cm,baseline=2.5ex,level 1/.style={level distance=0.4cm}]
\node [] {} [grow=up]
{	
	child {node [sq]  {}
	}
};
\end{tikzpicture} 
,
\begin{tikzpicture}
[  level distance=0.5cm, level 2/.style={sibling distance=0.6cm}, sibling distance=0.6cm,baseline=2.5ex,level 1/.style={level distance=0.4cm}]
\node [] {} [grow=up]
{	
	child {node [sq]  {}
		child {node [sq]  {} 
		}
	}
};
\end{tikzpicture} 
,
\begin{tikzpicture}
[  level distance=0.5cm, level 2/.style={sibling distance=0.6cm}, sibling distance=0.6cm,baseline=2.5ex,level 1/.style={level distance=0.4cm}]
\node [] {} [grow=up]
{	
	child {node [sq]  {}
		child {node [sq]  {} }
		child {node [sq]  {} 
		}
	}
};
\end{tikzpicture} 
,
\begin{tikzpicture}
[  level distance=0.5cm, level 2/.style={sibling distance=0.6cm}, sibling distance=0.6cm,baseline=2.5ex,level 1/.style={level distance=0.4cm}]
\node [] {} [grow=up]
{	
	child {node [sq]  {}
		child {node [sq]  {} 
			child {node [sq]  {} 
		}}
	}
};
\end{tikzpicture} 
,
\begin{tikzpicture}
[  level distance=0.5cm, level 2/.style={sibling distance=0.6cm}, sibling distance=0.6cm,baseline=2.5ex,level 1/.style={level distance=0.4cm}]
\node [] {} [grow=up]
{	
	child {node [sq]  {}
		child {node [sq]  {} }
		child {node [sq]  {} 
			child {node [sq]  {} }
		}
	}
};
\end{tikzpicture} 
,
\begin{tikzpicture}
[  level distance=0.5cm, level 2/.style={sibling distance=0.6cm}, sibling distance=0.6cm,baseline=2.5ex,level 1/.style={level distance=0.4cm}]
\node [] {} [grow=up]
{	
	child {node [sq]  {}
		child {node [sq]  {} 
			child {node [sq]  {} } }
		child {node [sq]  {} }
	}
};
\end{tikzpicture} 
,
\begin{tikzpicture}
[  level distance=0.5cm, level 2/.style={sibling distance=0.6cm}, sibling distance=0.6cm,baseline=2.5ex,level 1/.style={level distance=0.4cm}]
\node [] {} [grow=up]
{	
	child {node [my circle bis]  {}
		child {node [sq]  {} }
		child {node [sq]  {} 
			child {node [sq]  {} }
		}
	}
};
\end{tikzpicture} 
,
\begin{tikzpicture}
[  level distance=0.5cm, level 2/.style={sibling distance=0.6cm}, sibling distance=0.6cm,baseline=2.5ex,level 1/.style={level distance=0.4cm}]
\node [] {} [grow=up]
{	
	child {node [sq]  {}
		child {node [sq]  {} }
		child {node [sq]  {} }
		child {node [sq]  {} }
	}
};
\end{tikzpicture} 
, 
\begin{tikzpicture}
[  level distance=0.5cm, level 2/.style={sibling distance=0.6cm}, sibling distance=0.6cm,baseline=2.5ex,level 1/.style={level distance=0.4cm}]
\node [] {} [grow=up]
{	
	child {node [sq]  {}
		child {node [sq]  {} 
			child {node [sq]  {} 
				child {node [sq]  {} 
		}}}
	}
};
\end{tikzpicture} 
,\dots \bigg\}.
\end{equation*}
Let us distinguish the subset $\ca{G}_{\bullet}$ of those classes of trees that have at least one round-shape vertex (\ie the generating set of the Lie elements); let $\ca{G}_{\bsqsmall}$ be its complementary.

The operadic structure of $\TB$ provides both the Lie and the post-Lie product of any two elements. 
Explicitly, the Lie product of two generators $R$ and $S$ is the class of the tree $C(\bullet; R,S)$; their post-Lie product $R\triangleright S$ is as follows. 

\begin{defn}
	For a tree $T$ in $\ca{G}$, let $V_{\bullet}(T)$ and $V_{\bsqsmall}(T)$ be the sets of round-shape and square-shape vertices of $T$, respectively. 
\end{defn}

Suppose $R$ is a tree in $\ca{G}\setminus \ca{G}_{\bullet}$ and $S \in \ca{G}$. 
The post-Lie product of $R$ and $S$ is given by 
\begin{equation} \label{eq: grafting sq tree to any tree}
R\triangleright S = \sum_{v \in V_{\bsqsmall}(S)} R\triangleright_v S.
\end{equation}

Suppose that $R$ is in $\ca{G}_{\bullet}$. 
Recall from \cite[Section 3.3.1]{MQS} that transpositions act on each round-shape vertex of $R$ by switching its two outputs, and that each 
tuple of transpositions $\sigma \in  \mathbb{S}_2^{\times |V_{\bullet}(R)|}$ provides a tree $R_{\sigma}$ by performing such action vertex-wise. 
Recall also that $R$ (and also $R_{\sigma}$) can be \emph{contracted} into a tree $Con(R)$ with only one round-shape vertex (with possibly more than two outputs, so such a tree does not necessarily belong to $\ca{G}$); it is obtained by contracting all the edges between round-shape vertices. 
Given two vertices $v_1$ and $v_2$ of a tree $T$, we let $Con_{(v_1,v_2)}(T)$ be the tree obtained from $T$ by contracting the edge between $v_1$ and $v_2$; the resulting vertex inherits of the shape of $v_2$.  
One has
\begin{equation}\label{eq: grafting lie to any tree}
R\plprod S = \sum_{v \in V_{\bsqsmall}(S)} \sum_{\sigma\in \mathbb{S}_2^{\times |V_{\bullet}(R)|}} \epsilon(\sigma) Con_{(r,v)}(Con(R_{\sigma}) \plprod_v S),
\end{equation}
where $r$ is the root-vertex of $R$ and the sign $\epsilon(\sigma)$ is the product $sgn(\sigma_1)\cdots sgn(\sigma_k)$ for $\sigma=(\sigma_1,...,\sigma_k)$. 

Let us interpret $Con_{(r,v)}(Con(R_{\sigma}) \plprod_v S)$ in terms of grating of forests:  
If $R$ has $k$ round-shape vertices, it corresponds to a $k$--bracketing of trees $T_1,...,T_k$ in $\ca{G}_{\bsqsmall}$, so that one has 
\begin{equation}\label{eq: relation contraction and grafting forests}
Con_{(r,v)}(Con(R_{\sigma}) \plprod_v S) = (T_{\sigma(1)}T_{\sigma(2)}\cdots T_{\sigma(k)}) \plprod_v S.
\end{equation}

\subsection{The universal enveloping algebra of the free post-Lie algebra}\label{sec: univ env alg}

\begin{defn}
Let $n,k\geq 1$ and $q_1+...+q_k=n$ be a partition of $n$ by positive integers $q_i\geq 0$. 
A \emph{$(q_1,...,q_k)$--shuffle} is a partition of $\{1<\cdots <n\}$ by $k$ ordered sets of cardinal $q_i$ for each $1\leq i\leq k$.  
\end{defn}
The number of $(q_1,...,q_k)$--shuffles is $\shuff_{q_1,...,q_k}:=\frac{(q_1+...+q_k)!}{q_1!\cdots q_k!}$. 
\\
Recall that the universal enveloping  algebra of a post-Lie algebra (and in fact of any Lie algebra) is equipped with the shuffle coproduct
$\Delta_{sh}\co \mathcal{U}(\g) \to \mathcal{U}(\g)^{\ot 2}$ that makes it a bialgebra with respect to its classical product. 
Lie elements are primitive for $\Delta_{sh}$, that is one has  $\Delta_{sh}(l)  = l\ot 1 + 1\ot l$ for all $l\in \g$. 
Therefore, for any Lie element $l$ and $k\geq 2$, one has  
\begin{equation}\label{eq: shuffle copro on lie}
\Delta_{sh}^{(k)} (l)  =  \sum_{i_1+...+i_k=i} \shuff_{i_1,...,i_k} l^{i_1} \ot \cdots \ot l^{ i_k}.   
\end{equation}

Recall from Remark \ref{rmk: GL product and D-alg}, that 
the Grossman-Larson product on  $(\mathcal{U}(\g),\Delta_{sh})$ is given by $\ast\co X\ot Y \mapsto  X_{(1)}(X_{(2)}\plprod Y)$ for all $X,Y\in \ca{U}(\g)$. 
Recall also that here $\plprod \co \mathcal{U}(\g)^{\ot 2} \to \mathcal{U}(\g)$ is the extension of the post-Lie product and it satisfies the  properties \ref{D bial item1}, \ref{D bial item4} and \ref{D bial item5}.  
The left-side extension of the post-Lie product was identified in \cite{MQS} as  \emph{post-symmetric braces}, which are operations encoded by the corollas in $\PSB$. 
This means that, for $X=x_1\cdots  x_n \in \mathcal{U}(\g)$ and $y\in \g$, one has 
\begin{equation}\label{eq: braces}
X\plprod y = 
\begin{tikzpicture}
[  level distance=0.5cm, level 2/.style={sibling distance=0.6cm}, sibling distance=0.6cm,baseline=2.5ex,level 1/.style={level distance=0.4cm}]
\node [] {} [grow'=up]
{	
	child {node [my circle, label=right:\small{$n+1$}]  {}
		child {node [my circle,label=above:\small{$1$}]  {}} 
		child {node [my circle,label=above:\small{$2$}]  {}} 
		child {node [label=above:\small{$\dots$}]        {}} 
		child {node [my circle,label=above:\small{$n$}]  {}} 
	}
};
\end{tikzpicture}  
\left( x_1 \ot \cdots \ot x_n \ot y \right). 
\end{equation}

The rest of this section is dedicated to the universal enveloping algebra of the free post-Lie algebra $\TB(\K)$. 
Remark that, since the product of the free associative algebra on $\TB(\K)$ is given by concatenation of trees, its underlying vector space is generated by the forests of $\mathcal{G}$. 
The universal enveloping algebra   $(\mathcal{U}(\TB (\K)), \ast)$ is the vector space generated by the forests on $\mathcal{G}$,  modded out by the ideal generated by $RS-SR- C(\bullet; R,S)$ for every  $R$ and  $S$ in $\mathcal{G}$.  
For later use, let us investigate the product $E\ast F$ for a few particular forests $E$ and $F$.

\begin{lem}\label{lem: grafting of Lie element}
	For any Lie element $l$ and any forests $F$, one has $l \ast F = l F + l\plprod F $. 
\end{lem}
\begin{proof}
	Recall that Lie elements are primitive elements for the shuffle coproduct.  
	We conclude by observing that $l\ast F :=l_{(1)}\cdot (l_{(2)}\plprod F)$.  
\end{proof}

\begin{lem}\label{lem: calc 2 F=T}
	For $n\geq 1$ and $T$ a tree, one has 
	\begin{equation*}
	\bsq^{\times n} \plprod T = 
	\sum_{1\leq k \leq |T|}
	\sum_{ \substack{n_1+...+n_k=n,~ n_i>0 \\ \{v_1,...,v_k\},~ v_i\in T, v_i\neq v_j} } 
	\shuff_{n_1,...,n_k}~ 
	\bsq^{\times n} \plprod_{v_1,...,v_k}^{n_1,...,n_k} T. 
	\end{equation*}
\end{lem}
\begin{proof}
	Recall from \eqref{eq: braces} that $\bsq^{\times n} \plprod T$ is given by 
	\begin{equation*}
	\Big( \Big( \cdots \Big(
	\begin{tikzpicture}
	[  level distance=0.5cm, level 2/.style={sibling distance=0.6cm}, sibling distance=0.6cm,baseline=2.5ex,level 1/.style={level distance=0.4cm}]
	\node [] {} [grow'=up]
	{	
		child {node [my circle, label=right:\small{$n+1$}]  {}
			child {node [my circle,label=above:\small{$1$}]  {}} 
			child {node [my circle,label=above:\small{$2$}]  {}} 
			child {node [label=above:\small{$\dots$}]        {}} 
			child {node [my circle,label=above:\small{$n$}]  {}} 
		}
	};
	\end{tikzpicture}  
	\circ_{n+1} 
	T
	\Big)
	\circ_{1} 
	\bsq
	\Big) \circ_2 \cdots \Big)
	\circ_n \bsq.
	\end{equation*}
	From the operadic structure of $\PSB$, see \eqref{eq: explicit partial compo partial planar tree}, we know that this is the sum of $\bsq^{\times n} \plprod_{v_1,...,v_k}^{n_1,...,n_k} T$ over all distinct vertices $v_1,...,v_k$ and all the maps $\phi\co \{1,...,n\}\to \{1,...,k\}$ such that $|\phi(i)^{-1}|=n_i$ for $1\leq i\leq k$.   
\end{proof}

\begin{lem}\label{lem: calc 1 F=TA...Tk}
	Let $F=T_1\cdots T_k$ be a forest of $k$ trees and let $n\geq 1$. 
	One has 
	\begin{equation*}
	\bsq^{\times n} \ast F 
	= \sum_{j_0+...j_{k}=n, ~j_i\geq 0} 
	\shuff_{j_0,...,j_{k} }
	\bsq^{\times j_0} (\bsq^{\times j_1} \plprod T_1 )  \cdots (\bsq^{\times j_k} \plprod T_k).
	\end{equation*}
\end{lem}
\begin{proof}
	By \eqref{eq: shuffle copro on lie}, one has 
	\begin{equation*}
	\bsq^{\times n} \ast F = \sum_{j_0+j_1=n, ~j_i\geq 0} \shuff_{j_0,j_1} \bsq^{\times j_0}( \bsq^{\times j_1} \plprod F),
	\end{equation*}
	and in turn, by \ref{D bial item4} and \eqref{eq: shuffle copro on lie},   one has 
	\begin{equation*}
	\bsq^{\times i} \plprod F 
	= \bsq^{\times i} \plprod (T_1\cdots T_k) 
	= \sum_{i_1+...+i_k=i} \shuff_{i_1,...,i_k} (\bsq^{\times i_1} \plprod T_1 )  \cdots (\bsq^{\times i_k} \plprod T_k ).
	\end{equation*}	
\end{proof}

\subsection{Nested tubings}
In this subsection are presented two notions of nested tubings of forests, the \emph{vertical} nested tubings and the \emph{horizontal} ones. 
The term \emph{tubing} is borrowed from \cite{CarrDevadoss} though the present definition differs from the original one. 
\\

Recall the level partial order $\prec$ and the canonical linear order $<$ of Section \ref{sec: Operad of PRT}, given for trees. 
For a vertex $v$ of a tree $T$, let $\mathfrak{b}_v \subset V(T)$ be the subset of the $\prec$--predecessors of $v$; it inherits of the order $<$.  
The set of roots $\text{Root}(F)$ of a forest $F$ has a \emph{horizontal} order $<_h$ that is increasing as one goes from left to right: for a forest $ST$, one has $v<_h w$ for $v$ the root of $S$ and $w$ the root of $T$. 
 
Recall that $\mathcal{G}_{\bsqsmall}$ be is the subset of $\mathcal{G}$ of those trees that have only square-shape vertices $\bsq$. 
Let $\Forest$ be the set of the forests on $\mathcal{G}_{\bsqsmall}$; it admits a decomposition into the subsets $\Forest_n$ of those forests that have exactly $n$ vertices.  
Let $\Forest_n'$ be the set $\Forest_n \setminus \{\bsq^{\times n}\}$, where $\bsq^{\times n}$ denotes the {horizontal} forest of $n$ trees.  
\begin{defn}
	A \emph{higher set} of a poset $(\mathcal{P},<)$ is a subset of $\mathcal{P}$ that contains the $<$--successors of each of its elements.  
\end{defn}
\begin{defn}\label{de: tube}
	A  \emph{tube} of a tree $T$ is a connected higher set $t$ of $(V(T),\prec)$,  such that, for each $v\in t$, one has $t\cap \mathfrak{b}_v$ is a higher set  of $(\mathfrak{b}_v,<)$. 
\end{defn}
\begin{defn}
	A  \emph{tube} of a forest $F\in \Forest$ is a subset of $V(F)$ such that its intersection with $(\text{Root}(F),<_h)$ is a higher set and such  that it intersects each tree of $F$ into a (possibly empty) tube. 
\end{defn}

\begin{rem}
	A tube of a forest $F$ can be identify to a sub forest of $F$; we will often use this identification implicitly.  
\end{rem}

\begin{defn}\label{de: pre tubing}
	A  \emph{nested tubing} of $F\in \Forest$ is a collection of non empty tubes of $F$ that are pairwise nested and such that: 
	 	\begin{enumerate}
	 	\item it contains at least two tubes; \label{item: at least 2 tubes} 
	 	\item it contains the maximal tube (the tube that is the whole set of the vertices of the forest); \label{item: max tube}
	 \end{enumerate}
	For a nested tubing $t=\{t_i\}_{i\in I}$ of $F$, the \emph{boundary} of the tube $t_i$ is $\partial t_i= t_i \setminus \{ t_j\varsubsetneq t_i\}$. 
\end{defn}

\subsubsection{Vertical nested tubings}

\begin{defn}\label{de: tubing}
	A \emph{vertical} nested tubing is a nested tubing such that: 
	\begin{enumerate}
		\item\label{item: horizontal tubes}  the boundary of each tube is \emph{not} a horizontal forest of more than one tree; and,   
		\item\label{item: forest conneccted to one vertex} if the boundary of a tube is a forest, then either all the roots of this forest are connected to a single vertex of a sub tube, or none of the roots are connected to any sub tube. 
	\end{enumerate}
\end{defn}

\begin{defn}
For  $F$ in $\Forest'$, let $Tub(F)$ be the set of its vertical nested tubings.  
\end{defn}

\begin{example}\label{ex: tubings and not}
	a) is a tube, b) is a higher set that is not a tube (condition on $\mathfrak{b}_v$ unsatisfied), c), d) and e) are not vertical nested tubings (condition \ref{item: horizontal tubes} is not satisfied; in addition for d), condition \ref{item: forest conneccted to one vertex} is not satisfied either).  
	The three last examples f), g) and h) are vertical nested tubings. 
	\begin{equation*}
	\pgfdeclarelayer{foreground} 
	\pgfdeclarelayer{fforeground} 
	\pgfdeclarelayer{background}
	\pgfdeclarelayer{bbackground}
	\pgfsetlayers{bbackground,background,main,foreground,fforeground}
	a) 
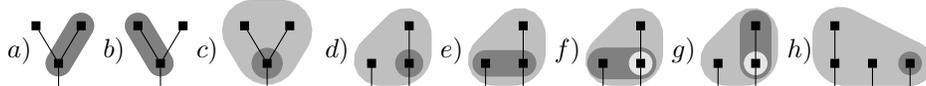
	\begin{tikzpicture}
	[level 1/.style={level distance=0cm,sibling distance=1.0cm }, level 2/.style={level distance=0.3cm,sibling distance=1.0cm }, level 3/.style={sibling distance=0.6cm,level distance=0.5cm}, sibling distance=0.6cm,baseline=2.5ex]
	\begin{pgfonlayer}{fforeground} 
	\node [] {} [grow=up]	 
	{	child  { edge from parent[draw=none]  child {node [sq] (A)  {} 	child {node [sq]  (B) {} 		} 		child {node [sq] (C) {}} 	 	}}
	};
	\end{pgfonlayer}
	\begin{pgfonlayer}{foreground}
	\draw [line width=10pt,opacity=1,black!50,line cap=round,rounded corners]  (A.center) -- (B.center);
	\end{pgfonlayer}
	\end{tikzpicture}
	~~~b) 
	\begin{tikzpicture}
	[level 1/.style={level distance=0cm,sibling distance=1.0cm }, level 2/.style={level distance=0.3cm,sibling distance=1.0cm }, level 3/.style={sibling distance=0.6cm,level distance=0.5cm}, sibling distance=0.6cm,baseline=2.5ex]
	\begin{pgfonlayer}{fforeground} 
	\node [] {} [grow=up]	 
	{	child  { edge from parent[draw=none]  child {node [sq] (A)  {} 	child {node [sq]  (B) {} 		} 		child {node [sq] (C) {}} 	 	}}
	};
	\end{pgfonlayer}
	\begin{pgfonlayer}{foreground}
	\draw [line width=10pt,opacity=1,black!50,line cap=round,rounded corners]  (A.center) -- (C.center);
	\end{pgfonlayer}
	\end{tikzpicture}
	~~~c) 
	\begin{tikzpicture}
	[level 1/.style={level distance=0cm,sibling distance=1.0cm }, level 2/.style={level distance=0.3cm,sibling distance=1.0cm }, level 3/.style={sibling distance=0.6cm,level distance=0.5cm}, sibling distance=0.6cm,baseline=2.5ex]
	\begin{pgfonlayer}{fforeground} 
	\node [] {} [grow=up]	 
	{	child  { edge from parent[draw=none]  child {node [sq] (A)  {} 	child {node [sq]  (B) {} 		} 		child {node [sq] (C) {}} 	 	}}
	};
	\end{pgfonlayer}
	\begin{pgfonlayer}{foreground}
	\draw [fill=black!50, draw=black!50, line cap=round,rounded corners] (A) circle (0.2);
	\end{pgfonlayer}
	\begin{pgfonlayer}{bbackground}
	\draw [line width=20pt,opacity=1,draw=black!25,fill=black!25,line cap=round,rounded corners] (A.center) -- (B.center) -- (C.center) --cycle;
	\end{pgfonlayer}
	\end{tikzpicture}
	~~~d) 
	\begin{tikzpicture}
	[level 1/.style={level distance=0cm,sibling distance=.5cm }, level 2/.style={level distance=0.3cm,sibling distance=1.0cm }, level 3/.style={sibling distance=0.6cm,level distance=0.5cm}, sibling distance=0.6cm,baseline=2.5ex]
	\begin{pgfonlayer}{fforeground} 
	\node [] {} [grow=up]	 
	{	child  { edge from parent[draw=none]  child {node [sq] (A)  {} 	child {node [sq]  (B) {} 		} 		}}
		child { edge from parent[draw=none] child  {	 {node [sq,black] (l1) {} 
				}
		}}
	};
	\end{pgfonlayer}
	\begin{pgfonlayer}{foreground}
	\draw [fill=black!50, draw=black!50, line cap=round,rounded corners] (A) circle (0.18);
	\end{pgfonlayer}
	\begin{pgfonlayer}{bbackground}
	\draw [line width=17pt,opacity=1,draw=black!25,fill=black!25,line cap=round,rounded corners] (B.center) -- (A.center)--(l1.center) --cycle;
	\end{pgfonlayer}
	\end{tikzpicture}
	~~~e)
	\begin{tikzpicture}
	[level 1/.style={level distance=0cm,sibling distance=.5cm }, level 2/.style={level distance=0.3cm,sibling distance=1.0cm }, level 3/.style={sibling distance=0.6cm,level distance=0.5cm}, sibling distance=0.6cm,baseline=2.5ex]
	\begin{pgfonlayer}{fforeground} 
	\node [] {} [grow=up]	 
	{	child  { edge from parent[draw=none]  child {node [sq] (A)  {} 	child {node [sq]  (B) {} 		} 		}}
		child { edge from parent[draw=none] child  {	 {node [sq,black] (l1) {} 
				}
		}}
	};
	\end{pgfonlayer}
	\begin{pgfonlayer}{main}
	\draw [line width=10pt,black!50, line cap=round,rounded corners] (l1.center)  -- (A.center);
	\end{pgfonlayer}
	\begin{pgfonlayer}{bbackground}
	\draw [line width=17pt,opacity=1,draw=black!25,fill=black!25,line cap=round,rounded corners] (A.center) -- (l1.center)--(B.center) --cycle;
	\end{pgfonlayer}
	\end{tikzpicture}
	~~~f)
	\begin{tikzpicture}
	[level 1/.style={level distance=0cm,sibling distance=.5cm }, level 2/.style={level distance=0.3cm,sibling distance=1.0cm }, level 3/.style={sibling distance=0.6cm,level distance=0.5cm}, sibling distance=0.6cm,baseline=2.5ex]
	\begin{pgfonlayer}{fforeground} 
	\node [] {} [grow=up]	 
	{	child  { edge from parent[draw=none]  child {node [sq] (A)  {} 	child {node [sq]  (B) {} 		} 		}}
		child { edge from parent[draw=none] child  {	 {node [sq,black] (l1) {} 
				}
		}}
	};
	\end{pgfonlayer}
	\begin{pgfonlayer}{foreground}
	\draw [fill=black!10, draw=black!10, line cap=round,rounded corners] (A) circle (0.15);
	\end{pgfonlayer}
	\begin{pgfonlayer}{main}
	\draw [line width=12pt,black!50, line cap=round,rounded corners] (l1.center)  -- (A.center);
	\end{pgfonlayer}
	\begin{pgfonlayer}{bbackground}
	\draw [line width=17pt,opacity=1,draw=black!25,fill=black!25,line cap=round,rounded corners] (B.center) -- (A.center)--(l1.center) --cycle;
	\end{pgfonlayer}
	\end{tikzpicture}
	~~~g)
	\begin{tikzpicture}
	[level 1/.style={level distance=0cm,sibling distance=.5cm }, level 2/.style={level distance=0.3cm,sibling distance=1.0cm }, level 3/.style={sibling distance=0.6cm,level distance=0.5cm}, sibling distance=0.6cm,baseline=2.5ex]
	\begin{pgfonlayer}{fforeground} 
	\node [] {} [grow=up]	 
	{	child  { edge from parent[draw=none]  child {node [sq] (A)  {} 	child {node [sq]  (B) {} 		} 		}}
		child { edge from parent[draw=none] child  {	 {node [sq,black] (l1) {} 
				}
		}}
	};
	\end{pgfonlayer}
	\begin{pgfonlayer}{foreground}
	\draw [fill=black!10, draw=black!10, line cap=round,rounded corners] (A) circle (0.15);
	\end{pgfonlayer}
	\begin{pgfonlayer}{main}
	\draw [line width=12pt,black!50, line cap=round,rounded corners] (B.center)  -- (A.center);
	\end{pgfonlayer}
	\begin{pgfonlayer}{bbackground}
	\draw [line width=17pt,opacity=1,draw=black!25,fill=black!25,line cap=round,rounded corners] (B.center) -- (A.center)--(l1.center) --cycle;
	\end{pgfonlayer}
	\end{tikzpicture}
	~~~h)
	\begin{tikzpicture}
	[level 1/.style={level distance=0cm,sibling distance=.5cm }, level 2/.style={level distance=0.3cm,sibling distance=1.0cm }, level 3/.style={sibling distance=0.6cm,level distance=0.5cm}, sibling distance=0.6cm,baseline=2.5ex]
	\begin{pgfonlayer}{fforeground} 
	\node [] {} [grow=up]	 
	{	child   { edge from parent[draw=none]   child {node [sq] (r)  {} 	 		} }	
		child   { edge from parent[draw=none]   child {node [sq] (m)  {} 	 		} }
		child { edge from parent[draw=none] child  {	 {node [sq,black] (l1) {}  child {node [sq]  (l2) {} 		}
				}
		}}
	};
	\end{pgfonlayer}
	\begin{pgfonlayer}{foreground}
	\draw [fill=black!50, draw=black!50, line cap=round,rounded corners] (r) circle (0.15);
	\end{pgfonlayer}
	\begin{pgfonlayer}{bbackground}
	\draw [line width=17pt,opacity=1,draw=black!25,fill=black!25,line cap=round,rounded corners] (r.center) --(l2.center) -- (l1.center)--(r.center);
	\end{pgfonlayer}
	\end{tikzpicture}
	\end{equation*}
\end{example}

\newcommand{\treeTT}{
	\begin{pgfonlayer}{fforeground} 
		\node [] {} [grow=up]	 
		{	
			child { edge from parent[draw=none]   child {node [sq] (RA)  {} 	child {node [sq]  (RB) {} 		} 		}}
			child { edge from parent[draw=none]   child {node [sq] (LA)  {} 	child {node [sq]  (LB) {} 		} 		}}
		};
	\end{pgfonlayer}
	\begin{pgfonlayer}{bbackground}
		\draw [line width=20pt,opacity=1,draw=black!25,fill=black!25,line cap=round,rounded corners] (LB.center) -- (LA.center)--(RA.center)  --(RB.center) --cycle;
	\end{pgfonlayer}
}

\newcommand{\treetwothree}{
\begin{pgfonlayer}{fforeground} 
\node [] {} [grow'=up]	 
{	
	child  {edge from parent[draw=none]  
		child {node [sq] (l1)  {}	
			child {node [sq] (l2)  {} 	
				child {node [sq] (l31) {} } 	 
				child {node [sq] (l32) {} } 	 	
				child {node [sq] (l33) {} } }}}
};
	\end{pgfonlayer}
\begin{pgfonlayer}{bbackground}
	\draw [line width=22pt,opacity=1,draw=black!25,fill=black!25,line cap=round,rounded corners] (l31.center)  --(l1.center) -- (l33.center) --  (l32.center) -- (l31.center) --(l1.center)--cycle;
\end{pgfonlayer}
}

\begin{example}
	Here below are all the vertical nested tubings of 
	\begin{tikzpicture}
	[level 1/.style={level distance=0cm,sibling distance=.5cm }, level 2/.style={level distance=0.3cm,sibling distance=1.0cm }, level 3/.style={sibling distance=0.5cm,level distance=0.5cm}, sibling distance=0.6cm,baseline=2.5ex]
	\node [] {} [grow=up]	 
	{	child  { edge from parent[draw=none]  child {node [sq] (A)  {} 	child {node [sq]  (B) {} 		} 		}}
		child { edge from parent[draw=none]   child {node [sq] (C)  {} 	child {node [sq]  (D) {} 		} 		}}		
	};
	\end{tikzpicture}.  %
	\begin{equation*}
	\pgfdeclarelayer{foreground} 
	\pgfdeclarelayer{fforeground} 
	\pgfdeclarelayer{background}
	\pgfdeclarelayer{bbackground}
	\pgfsetlayers{bbackground,background,main,foreground,fforeground}
	\begin{tikzpicture}
	[level 1/.style={level distance=0cm,sibling distance=.5cm }, level 2/.style={level distance=0.3cm,sibling distance=1.0cm }, level 3/.style={sibling distance=0.5cm,level distance=0.5cm}, sibling distance=0.6cm,baseline=2.5ex]
	\treeTT
	\begin{pgfonlayer}{foreground}
	\end{pgfonlayer}
	\begin{pgfonlayer}{main}
	\draw [line width=12pt,black!65, line cap=round,rounded corners] (RB.center)  -- (RA.center);
	\end{pgfonlayer}
	\end{tikzpicture}
	~
	\begin{tikzpicture}
	[level 1/.style={level distance=0cm,sibling distance=.5cm }, level 2/.style={level distance=0.3cm,sibling distance=1.0cm }, level 3/.style={sibling distance=0.5cm,level distance=0.5cm}, sibling distance=0.6cm,baseline=2.5ex]
	\treeTT
	\begin{pgfonlayer}{main}
	\draw [line width=12pt,black!65, line cap=round,rounded corners] (RB.center)  -- (RA.center) --(LA.center);
	\end{pgfonlayer}
	\end{tikzpicture}
	~
	\begin{tikzpicture}
	[level 1/.style={level distance=0cm,sibling distance=.5cm }, level 2/.style={level distance=0.3cm,sibling distance=1.0cm }, level 3/.style={sibling distance=0.5cm,level distance=0.5cm}, sibling distance=0.6cm,baseline=2.5ex]
	\treeTT
	\begin{pgfonlayer}{main}
	\draw [line width=12pt,black!65, line cap=round,rounded corners] (LB.center)  -- (LA.center) --(RA.center);
	\end{pgfonlayer}
	\end{tikzpicture}
	~
	\begin{tikzpicture}
	[level 1/.style={level distance=0cm,sibling distance=.5cm }, level 2/.style={level distance=0.3cm,sibling distance=1.0cm }, level 3/.style={sibling distance=0.5cm,level distance=0.5cm}, sibling distance=0.6cm,baseline=2.5ex]
	\treeTT
	\begin{pgfonlayer}{main}
	\draw [fill=black!40, draw=black!40, line cap=round,rounded corners] (RA) circle (0.12);
	\end{pgfonlayer}
	\begin{pgfonlayer}{background}
	\draw [line width=14pt,black!65, line cap=round,rounded corners] (RB.center)  -- (RA.center);
	\end{pgfonlayer}
	\end{tikzpicture}
	~
	\begin{tikzpicture}
	[level 1/.style={level distance=0cm,sibling distance=.5cm }, level 2/.style={level distance=0.3cm,sibling distance=1.0cm }, level 3/.style={sibling distance=0.5cm,level distance=0.5cm}, sibling distance=0.6cm,baseline=2.5ex]
	\treeTT
	\begin{pgfonlayer}{main}
	\draw [line width=9pt,black!40, line cap=round,rounded corners] (RB.center)  -- (RA.center) ;
	\end{pgfonlayer}
	\begin{pgfonlayer}{background}
	\draw [line width=14pt,black!65, line cap=round,rounded corners] (RB.center)  -- (RA.center) --(LA.center);
	\end{pgfonlayer}
	\end{tikzpicture}
	~
	\begin{tikzpicture}
	[level 1/.style={level distance=0cm,sibling distance=.5cm }, level 2/.style={level distance=0.3cm,sibling distance=1.0cm }, level 3/.style={sibling distance=0.5cm,level distance=0.5cm}, sibling distance=0.6cm,baseline=2.5ex]
	\treeTT
	\begin{pgfonlayer}{main}
	\draw [fill=black!40, draw=black!40, line cap=round,rounded corners] (RA) circle (0.12);
	\end{pgfonlayer}
	\begin{pgfonlayer}{background}
	\draw [line width=14pt,black!65, line cap=round,rounded corners] (LB.center)  -- (LA.center) --(RA.center);
	\end{pgfonlayer}
	\end{tikzpicture}
	~
	\begin{tikzpicture}
	[level 1/.style={level distance=0cm,sibling distance=.5cm }, level 2/.style={level distance=0.3cm,sibling distance=1.0cm }, level 3/.style={sibling distance=0.5cm,level distance=0.5cm}, sibling distance=0.6cm,baseline=2.5ex]
	\treeTT
	\begin{pgfonlayer}{foreground}
	\draw [fill=black!10, draw=black!10, line cap=round,rounded corners] (RA) circle (0.12);
	\end{pgfonlayer}
	\begin{pgfonlayer}{main}
	\draw [line width=9pt,black!40, line cap=round,rounded corners] (RB.center)  -- (RA.center) ;
	\end{pgfonlayer}
	\begin{pgfonlayer}{background}
	\draw [line width=14pt,black!65, line cap=round,rounded corners] (RB.center)  -- (RA.center) --(LA.center);
	\end{pgfonlayer}
	\end{tikzpicture}
	~
	\begin{tikzpicture}
	[level 1/.style={level distance=0cm,sibling distance=.5cm }, level 2/.style={level distance=0.3cm,sibling distance=1.0cm }, level 3/.style={sibling distance=0.5cm,level distance=0.5cm}, sibling distance=0.6cm,baseline=2.5ex]
	\treeTT
	\begin{pgfonlayer}{foreground}
	\draw [fill=black!10, draw=black!10, line cap=round,rounded corners] (RA) circle (0.11);
	\end{pgfonlayer}
	\begin{pgfonlayer}{main}
	\draw [line width=9pt,black!40, line cap=round,rounded corners] (LA.center) --(RA.center);
	\end{pgfonlayer}
	\begin{pgfonlayer}{background}
	\draw [line width=14pt,black!65, line cap=round,rounded corners] (RB.center)  -- (RA.center) --(LA.center);
	\end{pgfonlayer}
	\end{tikzpicture}
	~
	\begin{tikzpicture}
	[level 1/.style={level distance=0cm,sibling distance=.5cm }, level 2/.style={level distance=0.3cm,sibling distance=1.0cm }, level 3/.style={sibling distance=0.5cm,level distance=0.5cm}, sibling distance=0.6cm,baseline=2.5ex]
	\treeTT
	\begin{pgfonlayer}{foreground}
	\draw [fill=black!10, draw=black!10, line cap=round,rounded corners] (RA) circle (0.11);
	\end{pgfonlayer}
	\begin{pgfonlayer}{main}
	\draw [line width=9pt,black!40, line cap=round,rounded corners] (LA.center) --(RA.center);
	\end{pgfonlayer}
	\begin{pgfonlayer}{background}
	\draw [line width=14pt,black!65, line cap=round,rounded corners] (LB.center)  -- (LA.center) --(RA.center);
	\end{pgfonlayer}
	\end{tikzpicture}
	\end{equation*}
\end{example}

 \subsubsection{Horizontal nested tubings}

\begin{defn}
	A \emph{horizontal} nested tubing of $F\in Forest$ is a nested tubing of $F$ such that the boundary of each tube is a horizontal forest.  
\end{defn}

We let $hTub(F)$ denote the set of horizontal nested tubings of $F$. 
For any $p_1+...+p_k = N$ such that $p_i>0$, and $F\in \Forest'_N$, we let $hTub(F)_{p_1,...p_k}$ be the subset of $hTub(F)$ of those horizontal nested tubings $t=t_1\supset t_2 \supset \cdots \supset t_k$ such that $|\partial t_i|=p_i$ for each $1\leq i \leq k$.  
One has 
\begin{equation}\label{eq: decompo tubings}
hTub(F) = \bigsqcup_{p_1+...+p_k=N, p_i>0} hTub_{p_1,...,p_k}(F). 
\end{equation}

\begin{example}
	In Example \ref{ex: tubings and not}, a) is a tube whose boundary is not a horizontal forest, c) to g) are horizontal nested tubings, and h) is not horizontal. 
\end{example}

\begin{example}\label{example: horizontal tubings}
	Here below are all the horizontal nested tubings of 
	
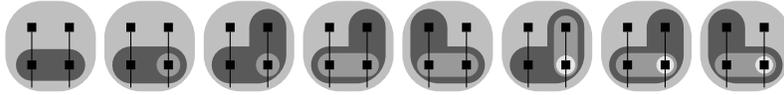
\begin{tikzpicture}
	[level 1/.style={level distance=0cm,sibling distance=.5cm }, level 2/.style={level distance=0.3cm,sibling distance=1.0cm }, level 3/.style={sibling distance=0.5cm,level distance=0.5cm}, sibling distance=0.6cm,baseline=2.5ex]
	\node [] {} [grow=up]	 
	{	child  { edge from parent[draw=none]  child {node [sq] (A)  {} 	child {node [sq]  (B) {} 		} 		}}
		child { edge from parent[draw=none]   child {node [sq] (C)  {} 	child {node [sq]  (D) {} 		} 		}}		
	};
	\end{tikzpicture}.  %
	\begin{equation*}
	\pgfdeclarelayer{foreground} 
	\pgfdeclarelayer{fforeground} 
	\pgfdeclarelayer{background}
	\pgfdeclarelayer{bbackground}
	\pgfsetlayers{bbackground,background,main,foreground,fforeground}
	\begin{tikzpicture}
	[level 1/.style={level distance=0cm,sibling distance=.5cm }, level 2/.style={level distance=0.3cm,sibling distance=1.0cm }, level 3/.style={sibling distance=0.5cm,level distance=0.5cm}, sibling distance=0.6cm,baseline=2.5ex]
	\treeTT
	\begin{pgfonlayer}{main}
	\draw [line width=12pt,black!65, line cap=round,rounded corners]  (RA.center) --(LA.center);
	\end{pgfonlayer}
	\end{tikzpicture}
	~
	\begin{tikzpicture}
	[level 1/.style={level distance=0cm,sibling distance=.5cm }, level 2/.style={level distance=0.3cm,sibling distance=1.0cm }, level 3/.style={sibling distance=0.5cm,level distance=0.5cm}, sibling distance=0.6cm,baseline=2.5ex]
	\treeTT
	\begin{pgfonlayer}{foreground}
	\draw [fill=black!40, draw=black!40, line cap=round,rounded corners] (RA) circle (0.15);
	\end{pgfonlayer}
	\begin{pgfonlayer}{background}
	\draw [line width=14pt,black!65, line cap=round,rounded corners] (RA.center) --(LA.center);
	\end{pgfonlayer}
	\end{tikzpicture}
	~
		\begin{tikzpicture}
	[level 1/.style={level distance=0cm,sibling distance=.5cm }, level 2/.style={level distance=0.3cm,sibling distance=1.0cm }, level 3/.style={sibling distance=0.5cm,level distance=0.5cm}, sibling distance=0.6cm,baseline=2.5ex]
	\treeTT
	\begin{pgfonlayer}{main}
	\draw [fill=black!40, draw=black!40, line cap=round,rounded corners] (RA) circle (0.15);
	\end{pgfonlayer}
	\begin{pgfonlayer}{background}
	\draw [line width=14pt,black!65, line cap=round,rounded corners] (RB.center)  -- (RA.center) --(LA.center);
	\end{pgfonlayer}
	\end{tikzpicture}
	~
	\begin{tikzpicture}
	[level 1/.style={level distance=0cm,sibling distance=.5cm }, level 2/.style={level distance=0.3cm,sibling distance=1.0cm }, level 3/.style={sibling distance=0.5cm,level distance=0.5cm}, sibling distance=0.6cm,baseline=2.5ex]
	\treeTT
	\begin{pgfonlayer}{main}
	\draw [line width=9pt,black!40, line cap=round,rounded corners] (LA.center) --(RA.center);
	\end{pgfonlayer}
	\begin{pgfonlayer}{background}
	\draw [line width=14pt,black!65, line cap=round,rounded corners] (RB.center)  -- (RA.center) --(LA.center);
	\end{pgfonlayer}
	\end{tikzpicture}
	~
	\begin{tikzpicture}
	[level 1/.style={level distance=0cm,sibling distance=.5cm }, level 2/.style={level distance=0.3cm,sibling distance=1.0cm }, level 3/.style={sibling distance=0.5cm,level distance=0.5cm}, sibling distance=0.6cm,baseline=2.5ex]
	\treeTT
	\begin{pgfonlayer}{main}
	\draw [line width=9pt,black!40, line cap=round,rounded corners] (LA.center) --(RA.center);
	\end{pgfonlayer}
	\begin{pgfonlayer}{background}
	\draw [line width=14pt,black!65, line cap=round,rounded corners] (LB.center)  -- (LA.center) --(RA.center);
	\end{pgfonlayer}
	\end{tikzpicture}
	~
	\begin{tikzpicture}
	[level 1/.style={level distance=0cm,sibling distance=.5cm }, level 2/.style={level distance=0.3cm,sibling distance=1.0cm }, level 3/.style={sibling distance=0.5cm,level distance=0.5cm}, sibling distance=0.6cm,baseline=2.5ex]
	\treeTT
	\begin{pgfonlayer}{foreground}
	\draw [fill=black!10, draw=black!10, line cap=round,rounded corners] (RA) circle (0.12);
	\end{pgfonlayer}
	\begin{pgfonlayer}{main}
	\draw [line width=9pt,black!40, line cap=round,rounded corners] (RB.center)  -- (RA.center) ;
	\end{pgfonlayer}
	\begin{pgfonlayer}{background}
	\draw [line width=14pt,black!65, line cap=round,rounded corners] (RB.center)  -- (RA.center) --(LA.center);
	\end{pgfonlayer}
	\end{tikzpicture}
	~
	\begin{tikzpicture}
	[level 1/.style={level distance=0cm,sibling distance=.5cm }, level 2/.style={level distance=0.3cm,sibling distance=1.0cm }, level 3/.style={sibling distance=0.5cm,level distance=0.5cm}, sibling distance=0.6cm,baseline=2.5ex]
	\treeTT
	\begin{pgfonlayer}{foreground}
	\draw [fill=black!10, draw=black!10, line cap=round,rounded corners] (RA) circle (0.11);
	\end{pgfonlayer}
	\begin{pgfonlayer}{main}
	\draw [line width=9pt,black!40, line cap=round,rounded corners] (LA.center) --(RA.center);
	\end{pgfonlayer}
	\begin{pgfonlayer}{background}
	\draw [line width=14pt,black!65, line cap=round,rounded corners] (RB.center)  -- (RA.center) --(LA.center);
	\end{pgfonlayer}
	\end{tikzpicture}
	~
	\begin{tikzpicture}
	[level 1/.style={level distance=0cm,sibling distance=.5cm }, level 2/.style={level distance=0.3cm,sibling distance=1.0cm }, level 3/.style={sibling distance=0.5cm,level distance=0.5cm}, sibling distance=0.6cm,baseline=2.5ex]
	\treeTT
	\begin{pgfonlayer}{foreground}
	\draw [fill=black!10, draw=black!10, line cap=round,rounded corners] (RA) circle (0.11);
	\end{pgfonlayer}
	\begin{pgfonlayer}{main}
	\draw [line width=9pt,black!40, line cap=round,rounded corners] (LA.center) --(RA.center);
	\end{pgfonlayer}
	\begin{pgfonlayer}{background}
	\draw [line width=14pt,black!65, line cap=round,rounded corners] (LB.center)  -- (LA.center) --(RA.center);
	\end{pgfonlayer}
	\end{tikzpicture}
	\end{equation*}
\end{example}

\subsection{Post-Lie Magnus expansion in terms of nested tubings}

For a post-Lie algebra $\g$, the pLMe of $x\in \g$  is the element $\chi(x)$ that satisfies the equation 
\begin{equation*}
\exp_{\cdot}(x)=\exp_{\ast}\big(\chi(x)\big).
\end{equation*} 

In particular, in $\hat{\mathcal{U}}_{\ast}(\TB(\K))$, it is a sum over all forests in $\Forest$:  
\begin{equation*}
\chi(\bsq)= \sum_{F\in \Forest} c_F F.
\end{equation*}  
We propose to compute the coefficients $c_F$ for any forest $F$ by two methods.

\subsubsection{Post-Lie Magnus expansion via vertical nested tubings}

As showed in \cite[Equation (81)]{KI} the pLMe can be expressed as a sum  
$\chi= \sum_{n\geq 1} \chi_n$, where $\chi_1(x)=x$ and  
\begin{equation*}
\chi_n(x)= \frac{x^n}{n!} - \sum_{\stackrel{k\geq 2, p_i>0}{p_1+...+p_k=n}} \frac{1}{k!} \chi_{p_1}(x)\ast \cdots \ast \chi_{p_k}(x) ~\text{ for all } x\in \g. 
\end{equation*} 

We will describe $\chi_n\co \TB (\K) \to \TB (\K)\subset \hat{\mathcal{U}_{\ast}}(\TB(\K))$ of the free post-Lie algebra on $\K=\K<\bsq>$. 

Note that $\chi_n(\bsq)$ is a homogeneous Lie polynomial of degree $n$. 
In particular, in $\hat{\mathcal{U}}_{\ast}(\TB(\K))$ it is  a sum over all forests in   $\Forest$   with $n$ vertices:  
\begin{equation*}
\chi_n(\bsq)= \sum_{F\in \Forest_n} c_F F.
\end{equation*}

\begin{defn}
	For $n,k\geq 2$, partition $p_1+...+p_k=n$ by strictly positive integers and $F$ in $\Forest'_n$, let $\mathfrak{D}(F)_{p_1,...,p_k}$ be the set of all the possible expressions  
	\begin{equation}\label{eq: graft and conc}
	F= F_1 \ltimes_1 (\cdots (\ltimes_2 (F_{k-2}  \ltimes (F_{k-1} \ltimes_{k-1} F_k))) \cdots ),  
	\end{equation}
	where $F_i$ runs through the forests of $p_i$ vertices that are not horizontal and $\ltimes_i$ is either the concatenation or the one vertex grafting operations $\plprod_v$ for some $v$. 
\end{defn}

\begin{lem}\label{lem: decompo}
	There is a bijection between $\mathfrak{D}(F)_{p_1,...,p_k}$ and the set of all the vertical nested tubings  $t=t_1\supset t_2 \supset \cdots \supset t_k$ of $F$ such that $|\partial t_i|=p_i$ for each $1\leq i \leq k$. 
\end{lem}
\begin{proof}
	Since concatenation and grafting do not remove vertices nor edges, the decomposition \eqref{eq: graft and conc} provides an embedding of $F_1,...,F_k$ into $F$, which we claim, can be represented by a vertical nested tubing. 
	Explicitly, the tube $t_k$ is $F_k$ seen in $F$ and is the most right sided subforest forest; the tube $t_{k-1}$ is sub forest $F_{k-1} \ltimes_{k-1} F_k$ of $F$ and it contains $F_k$, etc.  
	For example,   one has  
	
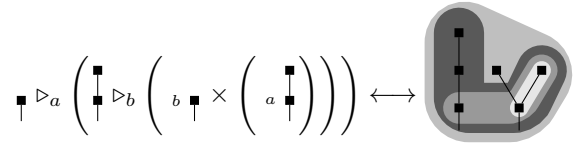
\begin{equation*}
	\begin{tikzpicture}
	[  level distance=0.5cm, level 2/.style={sibling distance=0.6cm}, sibling distance=0.6cm,baseline=2.5ex,level 1/.style={level distance=0.4cm}]
	\node [] {} [grow=up]
	{	
		child {node [sq]  {}
		}
	};
	\end{tikzpicture}  
	\triangleright_{a} 
	\left(
	\begin{tikzpicture}
	[  level distance=0.5cm, level 2/.style={sibling distance=0.6cm}, sibling distance=0.6cm,baseline=2.5ex,level 1/.style={level distance=0.4cm}]
	\node [] {} [grow=up]
	{	
		child {node [sq]  {}
			child {node [sq]  {} 
			}	
		}
	};
	\end{tikzpicture} 
	\triangleright_{b}
	\left(
	\begin{tikzpicture}
	[  level distance=0.5cm, level 2/.style={sibling distance=0.6cm}, sibling distance=0.6cm,baseline=2.5ex,level 1/.style={level distance=0.4cm}]
	\node [] {} [grow=up]
	{	
		child {node [sq,label=left:\tiny{$b$}]  {}	
		}
	};
	\end{tikzpicture} 
	\times 
	\left(
	\begin{tikzpicture}
	[  level distance=0.5cm, level 2/.style={sibling distance=0.6cm}, sibling distance=0.6cm,baseline=2.5ex,level 1/.style={level distance=0.4cm}]
	\node [] {} [grow=up]
	{	
		child {node [sq,label=left:\tiny{$a$}]  {}
			child {node [sq]  {} 
			}
		}
	};
	\end{tikzpicture}
	\right)\right)\right)
	\longleftrightarrow 
	\pgfdeclarelayer{foreground} 
	\pgfdeclarelayer{fforeground} 
	\pgfdeclarelayer{background}
	\pgfdeclarelayer{bbackground}
	\pgfsetlayers{bbackground,background,main,foreground,fforeground}
	\begin{tikzpicture}
	[level 1/.style={level distance=0cm,sibling distance=.8cm }, level 2/.style={level distance=0.3cm,sibling distance=1.0cm }, level 3/.style={sibling distance=0.6cm,level distance=0.5cm}, sibling distance=0.6cm,baseline=2.5ex]
	\begin{pgfonlayer}{fforeground} 
	\node [] {} [grow=up]	 
	{	child  { edge from parent[draw=none]  child {node [sq] (A)  {} 	child {node [sq]  (B) {} 		} 		child {node [sq] (C) {}} 	 	}}
		child { edge from parent[draw=none] child  {	 {node [sq,black] (l1) {} 
					child [black] {node [sq] (l2) {}  
						child {node [sq] (l3) {}} 
					}
				}
		}}
	};
	\end{pgfonlayer}
	\begin{pgfonlayer}{foreground}
	\draw [line width=7pt,opacity=1,black!10,line cap=round,rounded corners]  (A.center) -- (B.center);
	\end{pgfonlayer}
	\begin{pgfonlayer}{main}
	\draw [line width=12pt,black!40, line cap=round,rounded corners] (l1.center)  -- (A.center) -- (B.center);
	\end{pgfonlayer}
	\begin{pgfonlayer}{background}
	\draw [line width=19pt,black!65,line cap=round,rounded corners] (l3.center) -- (l2.center) -- (l1.center) --(A.east)   -- (B.center) ;
	\end{pgfonlayer}
	\begin{pgfonlayer}{bbackground}
	\draw [line width=26pt,opacity=1,black!25,line cap=round,rounded corners] (l1.center) -- (l2.center) -- (l3.center)     -- (B.center) -- (A.center) --cycle;
	\end{pgfonlayer}
	\end{tikzpicture}.
	\end{equation*}
	This assignment is well-defined: 
	\begin{itemize}
		\item grafting is on the left side of a vertex; this is condition on $\mathfrak{b}_v$ in Definition \ref{de: tube}; 
		\item one vertex grafting of forests corresponds to the condition \eqref{item: forest conneccted to one vertex} of Definition \ref{de: tubing}; 
		\item concatenations with right most parentheses correspond to the higher set condition for the order $<_h$. 
	\end{itemize}
	Let us show that such assignment is surjective. 
	Firstly, the above discussion shows that the tubings are such that they do not encode any other type of operations (than concatenations and one vertex graftings with right most parenthesis).  
	Secondly, note that condition \eqref{item: max tube} of Definition \ref{de: tubing} ensures that the whole forest is decomposed.
	Moreover, since the tubes $t_i$ are such that $|\partial t_i| = p_i$, their boundary $\partial t_i$ corresponds to sub forests of $F$ that are in $\Forest_{p_i}$. Finally, condition \eqref{item: horizontal tubes} of Definition \ref{de: tubing} corresponds to the absence of the forest $\bsq^{\times p}$.     
\end{proof}

\begin{prop}
	For $F\in \Forest'_n$, one has 
	$c_F= \sum_{t\in Tub(F)}  c_t$, where $c_t=\frac{-1}{|t|!} \prod_{t' \in t} c_{\partial t'}$ and $c_{\hspace{1pt}\bsq}=1$. 
\end{prop}
\begin{proof}
	As $\chi_1(\bsq)=\bsq$, the first coefficient $c_{\hspace{1pt}\bsq}$, which is the coefficient of the unique tubing of $\bsq$, is $1$. 
	Let $n\geq 2$ and $F$ be a forest in $\Forest_n$ that is not $\bsq^{\times n}$. To compute $c_F$, let us remark that $F$ appears in  
	$- \frac{1}{k!} \chi_{p_1}(x)\ast \cdots \ast \chi_{p_k}(x)$ for some $k$--partitions $p_1+...+p_k=n$. 
	For each of these partitions, $F$ is obtained  by shuffle concatenations and/or graftings of $k$ forests, say $F_1,....,F_k$, that belong to  $\chi_{p_1}(x)$,..., $\chi_{p_k}(x)$ respectively (shuffles arise from \ref{D bial item4}). 
	In fact, there are several operations of these types that we can exclude. 
	Indeed, since for all $p$ the element  $\chi_p(\bsq)$ is a Lie polynomial, every $F_i\in \chi_{p_i}$ for each $p_i$, is either a tree or belongs to a commutator. 
	Therefore, thanks to \eqref{eq: grafting sq tree to any tree} and \eqref{eq: grafting lie to any tree}-\eqref{eq: relation contraction and grafting forests} and to Lemma \ref{lem: grafting of Lie element}, when considering the product $F_i\ast F_{i+1}$ it is enough to consider the concatenation $F_iF_{i+1}$ and the grafting $F_i\plprod_v F_{i+1}$ for each vertex $v$ of $F_{i+1}$.     
	Moreover, since $\ast$ is associative, we can restrict ourselves to applying the concatenation and the one vertex grafting operations with the right most parentheses.  
	In other words, it is enough to consider all the expressions of the form \eqref{eq: graft and conc} which, thanks to Lemma \ref{lem: decompo}, correspond to vertical nested tubings. 	
	For each vertical nested tubings  $t=t_1\supset t_2 \supset \cdots \supset t_k$ of $F$ such that $|\partial t_i|=p_i$ for each $1\leq i \leq k$, 
	we let $c_t=\frac{-1}{k!}c_{F_1}c_{F_2}\cdots c_{F_k}$. 
	By summing over all the possible tubings, one obtains $c_F = \sum_{t\in Tub(F)} c_t$; note that condition \eqref{item: at least 2 tubes} ensures that tubings encode non-trivial decompositions. In addition, note that each $c_{F_i}$ itself is given by $c_{\partial t_i}$, which gives the result. 
	Note that Lemma \ref{lem: decompo} stands for forests $F_i$ in $\Forest'_{p_i}$, some of which may not be in $\chi_{p_i}(\bsq)$, that is, there are forests $F_i$ such that $c_{F_i}=0$. This is not an issue since if $c_{F_i}=0$ for some $i$, then $c_t=0$. 
\end{proof}

\begin{rem}
	The last  condition \eqref{item: horizontal tubes} of Definition \ref{de: tubing} may be removed, provided that one modifies Lemma \ref{lem: decompo} accordingly. 
	Indeed, since $c_{\bsqsmall^{\times n}}=0$ for $n\geq 2$, this does not interfere in the result.     
\end{rem}

\subsubsection{Post-Lie Magnus expansion via horizontal nested tubings}
This section is devoted to the computation of the pLMe using \emph{horizontal} nested tubings. 
While the previous method is recursive, the present method allows to compute the coefficient $c_F$ of any forest $F\in \Forest'_N$ for $N\geq 2$ in a closed form. 

We will use the following form of $\chi$: 
\begin{equation}\label{eq: chi = log exp}
\chi(x) = \log_*(\exp_{\cdot}(x)) = \sum_{k\geq 1} \frac{(-1)^{k-1}}{k} \left(\sum_{j\geq 1} \frac{x^j}{j!}\right)^{* k}. 
\end{equation}
In particular we will be led to investigate elements of the form 
\begin{equation}\label{eq: p& ast ...pk}
\bsq^{\times p_1} \ast ( \cdots \ast (\bsq^{\times p_{k-1}} \ast \bsq^{\times p_k})\cdots ), 
\end{equation} 
for partitions $p_1+...+p_k = N$ with $p_i>0$.

\begin{defn}
	For $N,k\geq 2$, partition $p_1+...+p_k=N$ by strictly positive integers and $F$ in $\Forest'_N$, let $h\mathfrak{D}(F)_{p_1,...,p_k}$ be the set of all 
	the possible expressions  of $F$ of the form 
	\begin{equation}\label{eq: express mixed oper}
	F = \bsq^{\times p_1} \ltimes_1 (\cdots (\ltimes_2 ( \bsq^{\times p_{k-2}}  \ltimes (  \bsq^{\times p_{k-1}} \ltimes_{k-1}  \bsq^{\times p_k} ))) \cdots ),  
	\end{equation}
	in which each  $\ltimes_i$ is an operation of the form $\ltimes_{v_1,...,v_k}^{n_0, n_1,...,n_k}$ as introduced in Notation \ref{notation: operations graft}, item \eqref{item: operations grafting/concat}. 
\end{defn}

\begin{lem}\label{lem: surjective map decompo to htubings}
	For each $N\geq 2$ and each $p_1+...+p_k=N, p_i>0$, there is a bijection between $h\mathfrak{D}(F)_{p_1,...,p_k}$ and $hTub_{p_1,...,p_k}(F)$.  
\end{lem}
\begin{proof}
Since the operations $\ltimes_{v_1,...,v_k}^{n_0,n_1,...,n_k}$ do not remove vertices nor edges, the decomposition \eqref{eq: express mixed oper} provides an embedding of $\bsq^{\times p_1},..., \bsq^{\times p_k}$ into $F$, which we claim, can be represented by a horizontal nested tubing. Explicitly, the tube $t_k$ is $\bsq^{\times p_k}$ seen in $F$ as the most right sided sub forest; the tube $t_{k-1}$ is the sub forest $\bsq^{\times p_{k-1}} \ltimes_{k-1}  \bsq^{\times p_k}$ that contains $\bsq^{\times p_k}$ and $\bsq^{\times p_{k-1}}$, etc.  
For example,  one has  
	\begin{equation}\label{eq: corresp decop tubing 2 }
		\begin{tikzpicture}
	[level 1/.style={level distance=0cm,sibling distance=.5cm }, level 2/.style={level distance=0.3cm,sibling distance=1.0cm }, level 3/.style={sibling distance=0.65cm,level distance=0.55cm}, sibling distance=0.6cm,baseline=2.5ex]
	\node [] {} [grow'=up]	 
	{	
		child {edge from parent[draw=none]  
		child {node [sq] (A) {}}
		}	
		child {edge from parent[draw=none]  
		child {node [sq] (B) {}	}
		}
  		child {edge from parent[draw=none]  
		child {node [sq] (B) {}	}
		}
		child {edge from parent[draw=none]  
			child {node [sq] (B) {}	}
		}
		child {edge from parent[draw=none]  
			child {node [sq] (B) {}	}
		}
	};
	\end{tikzpicture} 
	\ltimes_{v_1,v_2}^{2,2,1}
	\left(
		\begin{tikzpicture}
	[level 1/.style={level distance=0cm,sibling distance=.5cm }, level 2/.style={level distance=0.3cm,sibling distance=1.0cm }, level 3/.style={sibling distance=0.65cm,level distance=0.55cm}, sibling distance=0.6cm,baseline=2.5ex]
	\node [] {} [grow'=up]	 
	{	
		child {edge from parent[draw=none]  
		child {node [sq,label=above:\small{$v_2$}] (A) {}}
		}	
		child {edge from parent[draw=none]  
		child {node [sq] (B) {}	}
		}
	};
	\end{tikzpicture} 
	\plprod_{v_1}^2 
	\left(
	\begin{tikzpicture}
	[level 1/.style={level distance=0cm,sibling distance=.5cm }, level 2/.style={level distance=0.3cm,sibling distance=1.0cm }, level 3/.style={sibling distance=0.65cm,level distance=0.55cm}, sibling distance=0.6cm,baseline=2.5ex]
	\node [] {} [grow'=up]	 
	{	
		child {edge from parent[draw=none]  
		child {node [sq,label=above:\small{$v_1$}] (A) {}}
		}	
		child {edge from parent[draw=none]  
		child {node [sq] (B) {}	}
		}
	};
	\end{tikzpicture} 
	\right) \right)
	\mapsto
	\pgfdeclarelayer{foreground} 
	\pgfdeclarelayer{fforeground} 
	\pgfdeclarelayer{background}
	\pgfdeclarelayer{bbackground}
	\pgfsetlayers{bbackground,background,main,foreground,fforeground}
\begin{tikzpicture}
[level 1/.style={level distance=0cm,sibling distance=1.4cm }, level 2/.style={level distance=0.3cm,sibling distance=1.0cm }, level 3/.style={sibling distance=0.65cm,level distance=0.55cm}, sibling distance=0.6cm,baseline=2.5ex]
\begin{pgfonlayer}{fforeground} 
\node [] {} [grow'=up]	 
{	
	child {edge from parent[draw=none]  
		child {node [sq] (A) {}}
	}	
	child {edge from parent[draw=none]  
		child {node [sq] (B) {}	}
	}
	child {edge from parent[draw=none] 
		child {{node [sq,black] (T1lev0) {} 
				child {node [sq] (T1lev11) {}}  
				child {node [sq] (T1lev12) {}}  
				child {node [sq] (T1lev13) {} 
					child {node [sq] (T1lev23) {}} } 
				child {node [sq] (T1lev14) {}}  
			}
	}}
	child {edge from parent[draw=none]  
		child {node [sq] (T2) {}	}
	}
};
\end{pgfonlayer}
\begin{pgfonlayer}{foreground}
\draw [line width=10pt,opacity=1,black!10,line cap=round,rounded corners]  (T1lev0.center) -- (T2.center);
\end{pgfonlayer}
\begin{pgfonlayer}{background}
\draw [line width=19pt,black!65,line cap=round,rounded corners]  (T2.center) --(T1lev0.center) --(T1lev13.center) -- (T1lev14.center) --cycle;
\end{pgfonlayer}
\begin{pgfonlayer}{bbackground}
\draw [line width=27pt,opacity=1,black!20,line cap=round,rounded corners] (A.center) -- (B.center) -- (T1lev0.center)      --(T2.center) -- (T1lev14.center) --(T1lev23.center) --(T1lev11.center) --(A.center) ;
\end{pgfonlayer}
\end{tikzpicture}.
\end{equation}
Note that by construction the boundary of each tube is a horizontal forest, the maximal tube is included, and because decompositions are not trivial there are at least $2$ tubes. 
Moreover,  
	\begin{itemize}
			\item since grafting is on the left side of a vertex, the tubes satisfy condition on $\mathfrak{b}_v$ in Definition \ref{de: tube}; 
			\item since operations are performed with right most parentheses, the tubes satisfy the higher set condition for the order $<_h$. 
		\end{itemize}
Therefore, the map is well-defined. 

	The bijectivity can be shown by considering the inverse map, which is as follows.   
	Each pair of tubes $(t_i,t_{i-1})$ determines an operation  of the form $\ltimes_{v_1,...,v_k}^{n_0, n_1,...,n_k}$: 
	the integer $n_0$ is the number of roots of $\partial t_i$ that are not attached to any vertex of $t_{i-1}$ (they are the most left sided roots of $\partial t_i$ because of the higher set condition for $<_h$);  
	the integer $n_1$ is the number of next roots (from left to right) of $\partial t_i$ that are attached to a same vertex, which is $v_1$, etc. 
\end{proof}

\begin{notation}\label{notation th2}
	Let $F$ be a forest and $t$ a horizontal nested tubing of $F$. 
	For each non minimal tube $s$ of $t$ we let $s'\subset s$ be its predecessor in $t$; so, $\partial s = s\setminus s'$. 
	Recall that $s'$ is a forest, say of trees $T_1,..., T_{lg(s')}$. 
	\begin{itemize}
		\item 	We let $j(s)_0\geq 0$ be the number of roots of $F$ in $\partial s$, that is the number of vertices that are not attached to any vertex of $s'$. 
		\item For $1\leq a \leq lg(s')$, we let $j(s)_a$ be the number of vertices of $\partial s$ that are attached to vertices of $T_a$.
		\item For each tree $T_a$ of $s'$ and each vertex $v$ of $T_a$, we let $f_{T_a}^s(v)$ be the cardinal of its fiber in $\partial s$, that is the cardinal of the set $\mathfrak{b}_v\cap \partial s$ of vertices in $\partial s$ that are attached to $v$. 
		\item We let $k(T_a)\geq 0$ be the number of vertices $v$ of $T_a$ such that $f_{T_a}^s(v)\neq 0$ and we let $v_1,...,v_{k(T_a)}$ be the collection of such vertices.   
	\end{itemize}
\end{notation}

\begin{example}
Consider the horizontal nested tubing of \eqref{eq: corresp decop tubing 2 }; let $s$ be the maximal tube. 
The boundary $\partial s$ is a forest of five trees; the first two trees (\ie left most sided) are not attached to $s'$, and the next three trees are attached to the same tree of $s'$. Therefore one has $j(s)_0=2$,  $j(s)_1=3$ and $j(s)_2=0$. 
For the first tree $T_1$ of $s'$ (the corolla with 3 vertices, the root $v_r$, the most left-sided vertex $v_1$ and the other one $v_2$), one has $f^s_{T_1}(v_r)=2$, $f^s_{T_1}(v_1)=1$  and  $f^s_{T_1}(v_2)=0$.  
\end{example}

	For $F \in \Forest'_N$ and $t\in hTub(F)_{p_1,...,p_k}$,  
	we let $A(t)$ be the number of times the expression that corresponds to $t$ via Lemma \ref{lem: surjective map decompo to htubings} appears in 
	$\bsq^{\times p_1} \ast ( \cdots \ast (\bsq^{\times p_{k-1}} \ast \bsq^{\times p_k})\cdots )$. 
	
\begin{lem}\label{lem: decompo appears A times} 
	\begin{equation*}
	A(t) = \prod_{s\in t,~ s \text{ not minimal}} \shuff_{j(s)_0,...,j(s)_{lg(s')}}
		\prod_{1\leq a \leq lg(s')}
		\shuff_{ f^s_{T_a}(v_1),...,f^s_{T_{a}}(v_{k(T_a)}) }.
		\end{equation*}
	\end{lem}
\begin{proof}
	The proof is by induction. We let $t= t_k\supset t_{k-1} \supset \cdots \supset t_1$.  
	Consider $t_2\supset t_1$ as a horizontal nested tubing for the forest $t_2$. 
	Recall that $lg(t_i)$ is the number of trees in the tube $t_i$. 
	One has $lg(t_1)= p_1$ and we write $T_1\cdots T_{p_1}$ the decomposition of $t_1$ into trees. 
	Note that $t_2\supset t_1$ corresponds to 
	$\bsq^{\times p_{2}} \ltimes_{v_{i_1},...,v_{i_r}}^{j(t_2)_0,...,j(t_2)_{p_1}} \bsq^{\times p_{1}}$ 
	for some subset $\{i_1,...,i_r\}$ of $\{1,...,p_1\}$ where 	$v_i$ is the unique vertex of $T_i$. 
	By Lemma \ref{lem: calc 1 F=TA...Tk} and Lemma \ref{lem: calc 2 F=T},  the expression 
    $\bsq^{\times p_{2}} \ltimes_{v_{i_1},...,v_{i_r}}^{j(t_2)_0,...,j(t_2)_{p_1}} \bsq^{\times p_{1}}$ appears 
	$\shuff_{ j(t_2)_0,...,j(t_2)_{p_1} }
	\prod_{1\leq a \leq p_1}
	\shuff_{ f^s_{T_a}(v_1),...,f^s_{T_{a}}(v_{k(T_a)}) }$ times. 
	
	Suppose the statement is true for the tubing $t_j\supset t_{j-1} \supset \cdots \supset t_1$ of $t_j$, for any $2 \leq j < k$. 
	Let $s= t_k$ and $s'=t_{k-1}$. 
	Let $T_1\cdots T_{lg(s')}$ be the decomposition of $s'$ into trees. 
	By  Lemma \ref{lem: calc 1 F=TA...Tk} and Lemma \ref{lem: calc 2 F=T},  the expression  
	$\bsq^{\times p_{k}} \ltimes_{v_{i_1},...,v_{i_r}}^{j(s)_0,...,j(s)_{lg(s')}} s'$ appears 
	\begin{equation*}
	\shuff_{j(t_2)_0,...,j(t_2)_{p_1}}
	\prod_{1\leq a \leq p_1}
	\shuff_{ f^s_{T_a}(v_1),...,f^s_{T_{a}}(v_{k(T_a)}) }
	\end{equation*}
	multiplicate by the number of times the sub-expression that corresponds to $t_{k-1}\supset t_{k-2} \supset \cdots \supset t_1$ appears. 
	Hence the result.    
\end{proof}

\begin{thm}\label{th: decompo2}
	With Notation \ref{notation th2}, for each $F$ in $\Forest_N'$ with $N\geq 2$, the coefficient $c_F$ is 
	\begin{equation*}
	\sum_{t\in hTub(F)} \frac{(-1)^{|t|-1}}{|t|}  
	\prod_{s\in t}  \frac{1}{|\partial s|!}   
	\shuff_{j(s)_0,...,j(s)_{lg(s')}}
	\prod_{1\leq a \leq lg(s')}
	\shuff_{ f^s_{T_a}(v_1),...,f^s_{T_{a}}(v_{k(T_a)}) },
	\end{equation*}
where $\shuff_{j(s)_0,...,j(s)_{lg(s')}}
	\prod_{1\leq a \leq lg(s')}
	\shuff_{ f^s_{T_a}(v_1),...,f^s_{T_{a}}(v_{k(T_a)}) }:=1$ when $s$ is minimal.
\end{thm}
\begin{proof} 
Let $F$ be in $\Forest_N'$ for $N\geq 2$. 
By using equation \eqref{eq: chi = log exp}, one can write $\chi(\bsq)$ as 
\begin{equation*}
\sum_{N\geq 1}  
\sum_{p_1+...+p_k = N, ~p_i>0} \frac{(-1)^{k-1}}{k} \frac{1}{p_1!p_2!\cdots p_k!}  \bsq^{\times p_1} \ast ( \cdots \ast (\bsq^{\times p_{k-1}} \ast \bsq^{\times p_k})\cdots ).
\end{equation*} 
If we let $D_{p_1,...,p_k}$ denote the number of times the forest $F$ appears in $ \bsq^{\times p_1} \ast ( \cdots \ast (\bsq^{\times p_{k-1}} \ast \bsq^{\times p_k})\cdots )$, then one has $c_F= \sum_{p_1+...+p_k = N, ~p_i>0} \frac{(-1)^{k-1}}{k} \frac{1}{p_1!p_2!\cdots p_k!} D_{p_1,...,p_k}$. 

Let us compute $D_{p_1,...,p_k}$. 
Consider a decomposition of $F$ of the form \eqref{eq: express mixed oper}; by Lemma \ref{lem: surjective map decompo to htubings} this amounts to considering $t\in hTub(F)_{p_1,...,p_k}(F)$. 
Such a decomposition appears exactly $A(t)$ times in $\bsq^{\times p_1} \ast ( \cdots \ast (\bsq^{\times p_{k-1}} \ast \bsq^{\times p_k})\cdots )$. 
Therefore, one has $D_{p_1,...,p_k}= \sum_{t\in hTub(F)_{p_1,...,p_k}}     A(t)$, which by Lemma  \ref{lem: decompo appears A times}, gives  \begin{equation*}
D_{p_1,...,p_k}=
\sum_{t\in hTub(F)_{p_1,...,p_k}}   
\prod_{s\in t} 
\shuff_{j(s)_0,...,j(s)_{lg(s')}}
\prod_{1\leq a \leq lg(s')}
\shuff_{ f^s_{T_a}(v_1),...,f^s_{T_{a}}(v_{k(T_a)}) }.
\end{equation*}
Finally, using the decomposition \eqref{eq: decompo tubings} of $hTub(F)$ one obtains the result. 
\end{proof}

\begin{example}
	Here is presented the computation of $c_F$ for  
	$F=$
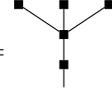
	\begin{tikzpicture}
	[level 1/.style={level distance=0cm,sibling distance=.5cm }, level 2/.style={level distance=0.3cm,sibling distance=1.0cm }, level 3/.style={sibling distance=0.6cm,level distance=0.4cm}, sibling distance=0.6cm,baseline=2.5ex]
	\node [] {} [grow'=up]	 
	{	
		child  {edge from parent[draw=none]  
			child {node [sq] (l1)  {}	
				child {node [sq] (l2)  {} 	
					child {node [sq] (l31) {} } 	 
					child {node [sq] (l32) {} } 	 	
					child {node [sq] (l33) {} } }}}
	};
	\end{tikzpicture}. 
	
	Let us list all the possible horizontal nested tubings, which are of the form $(p_k,p_{k-1},...,p_1)$ for $1\leq k\leq 5$. 
	There are only four possibilities, which corresponds to $(1,1,1,1,1)$, $(1,2,1,1)$, $(2,1,1,1)$ and $(3,1,1)$: 
	\begin{equation*}
	\pgfdeclarelayer{foreground} 
	\pgfdeclarelayer{fforeground} 
	\pgfdeclarelayer{fbforeground} 
	\pgfdeclarelayer{background}
	\pgfdeclarelayer{bbackground}
	\pgfsetlayers{bbackground,background,main,foreground,fbforeground,fforeground}
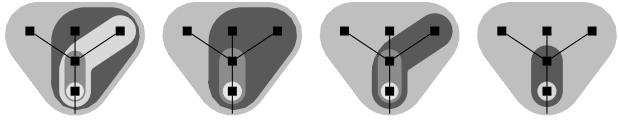
	\begin{tikzpicture}
	[level 1/.style={level distance=0cm,sibling distance=.5cm }, level 2/.style={level distance=0.3cm,sibling distance=1.0cm }, level 3/.style={sibling distance=0.6cm,level distance=0.4cm}, sibling distance=0.6cm,baseline=2.5ex]
	\treetwothree
	\begin{pgfonlayer}{bbackground}
	\draw [line width=22pt,opacity=1,draw=black!25,fill=black!25,line cap=round,rounded corners] (l31.center)  --(l1.center) -- (l33.center) --  (l32.center) -- (l31.center) --(l1.center)--cycle;
	\end{pgfonlayer}
	\begin{pgfonlayer}{background}
	\draw [line width=18pt,opacity=1,draw=black!65,fill=black!65,line cap=round,rounded corners] (l1.center) -- (l33.center) --  (l32.center)  --(l1.center)--cycle;
	\end{pgfonlayer}
	\begin{pgfonlayer}{main}
	\draw [line width=12pt,opacity=1,draw=black!15,fill=black!15,line cap=round,rounded corners] (l1.center) -- (l2.center) --(l33.center) ;
	\end{pgfonlayer}
	\begin{pgfonlayer}{foreground}
	\draw [line width=8pt,opacity=1,draw=black!45,fill=black!45,line cap=round,rounded corners] (l1.center) -- (l2.center) ;
	\end{pgfonlayer}
	\begin{pgfonlayer}{fbforeground}
	\draw [fill=black!10, draw=black!10, line cap=round,rounded corners] (l1) circle (0.11);
	\end{pgfonlayer}
	\end{tikzpicture}
	~~
	\begin{tikzpicture}
	[level 1/.style={level distance=0cm,sibling distance=.5cm }, level 2/.style={level distance=0.3cm,sibling distance=1.0cm }, level 3/.style={sibling distance=0.6cm,level distance=0.4cm}, sibling distance=0.6cm,baseline=2.5ex]
	\treetwothree
		\begin{pgfonlayer}{background}
	\draw [line width=18pt,opacity=1,draw=black!65,fill=black!65,line cap=round,rounded corners] (l1.center) -- (l33.center) --  (l32.center)  --(l1.center)--cycle;
	\end{pgfonlayer}
	\begin{pgfonlayer}{foreground}
	\draw [line width=10pt,opacity=1,draw=black!45,fill=black!45,line cap=round,rounded corners] (l1.center) -- (l2.center) ;
	\end{pgfonlayer}
	\begin{pgfonlayer}{fbforeground}
	\draw [fill=black!10, draw=black!10, line cap=round,rounded corners] (l1) circle (0.12);
	\end{pgfonlayer}
	\end{tikzpicture}
	~~
	\begin{tikzpicture}
	[level 1/.style={level distance=0cm,sibling distance=.5cm }, level 2/.style={level distance=0.3cm,sibling distance=1.0cm }, level 3/.style={sibling distance=0.6cm,level distance=0.4cm}, sibling distance=0.6cm,baseline=2.5ex]
	\treetwothree
		\begin{pgfonlayer}{main}
	\draw [line width=13.5pt,opacity=1,draw=black!65,fill=black!65,line cap=round,rounded corners] (l1.center) -- (l2.center)   --(l33.center);
	\end{pgfonlayer}
	\begin{pgfonlayer}{foreground}
	\draw [line width=9pt,opacity=1,draw=black!45,fill=black!45,line cap=round,rounded corners] (l1.center) -- (l2.center) ;
	\end{pgfonlayer}
	\begin{pgfonlayer}{fbforeground}
	\draw [fill=black!10, draw=black!10, line cap=round,rounded corners] (l1) circle (0.11);
	\end{pgfonlayer}
	\end{tikzpicture}
	~~
		\begin{tikzpicture}
	[level 1/.style={level distance=0cm,sibling distance=.5cm }, level 2/.style={level distance=0.3cm,sibling distance=1.0cm }, level 3/.style={sibling distance=0.6cm,level distance=0.4cm}, sibling distance=0.6cm,baseline=2.5ex]
	\treetwothree
		\begin{pgfonlayer}{foreground}
	\draw [line width=12pt,opacity=1,draw=black!65,fill=black!65,line cap=round,rounded corners] (l1.center) -- (l2.center) ;
	\end{pgfonlayer}
	\begin{pgfonlayer}{fbforeground}
	\draw [fill=black!15, draw=black!15, line cap=round,rounded corners] (l1) circle (0.12);
	\end{pgfonlayer}
	\end{tikzpicture}
	\end{equation*}
	In all those cases there are no shuffles involved because all tubes are trees and there is only one vertex which has a non trivial fiber. 
	One obtains 
	\begin{equation*}
	c_F= \frac{1}{5} (1\times 1 \times 1 \times 1\times 1) - \frac{1}{4} (1\times 1 \times \frac{1}{2!} \times 1 + 1\times 1 \times 1 \times  \frac{1}{2!}) + \frac{1}{3} (1\times 1 \times \frac{1}{3!}) = \frac{1}{180}.   
	\end{equation*} 
\end{example}
\begin{example}
	Here is the computation of $c_F$ for $F=$
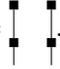
	\begin{tikzpicture}
	[level 1/.style={level distance=0cm,sibling distance=.5cm }, level 2/.style={level distance=0.3cm,sibling distance=1.0cm }, level 3/.style={sibling distance=0.5cm,level distance=0.5cm}, sibling distance=0.6cm,baseline=2.5ex]
	\node [] {} [grow=up]	 
	{	child  { edge from parent[draw=none]  child {node [sq] (A)  {} 	child {node [sq]  (B) {} 		} 		}}
		child { edge from parent[draw=none]   child {node [sq] (C)  {} 	child {node [sq]  (D) {} 		} 		}}		
	};
	\end{tikzpicture}. 
	The horizontal nested tubings are listed in Example \ref{example: horizontal tubings} and correspond to 
	$(2,2)$, $(2,1,1)$,  $(1,2,1)$, $(1,1,2)$ (two possible horizontal tubings), and $(1,1,1,1)$ (three possibilities).   
	In the first tubing, for $s$ the maximal tube, its predecessor $s'$ has two trees $T_1$ and $T_2$; one has $j(s)_0 = 0$ (there is no unattached vertices in $s$), $j(s)_1=1$ and $j(s)_2=1$; and, $f^s_{T_1}(v)= 1$ and $f^s_{T_2}(v_1)=1$. 
	Therefore $\shuff_{j(s)_0,...,j(s)_{lg(s')}}
	\prod_{1\leq a \leq lg(s')}
	\shuff_{ f^s_{T_a}(v_1),...,f^s_{T_{a}}(v_{k(T_a)}) } = \shuff_{1,1} \times \shuff_{1} \times \shuff_{1}= 2$. 
	 Doing this for each tube and tubing, one obtains, 
	\begin{equation*}
	c_F= -\frac{1}{2} (\frac{1}{2!}\times\frac{1}{2!}\times 2) + 
	\frac{1}{3} (\frac{1}{2!}\times 2  + \frac{1}{2!}\times 2 + \frac{1}{2!} +\frac{1}{2!}) 
	-\frac{1}{4} (1 + 1 + 1) = 0.   
	\end{equation*} 
	Of course, since $\chi(\bsq)$ is a Lie element, we already knew that $c_F=0$. 
\end{example}

To end this part we give the first four terms of $\chi(\bsq)$:   
\begin{align*}
\chi_1(\bsq) &= 
\begin{tikzpicture}
[level 1/.style={level distance=0cm,sibling distance=.5cm }, level 2/.style={level distance=0.3cm,sibling distance=1.0cm }, level 3/.style={sibling distance=0.6cm,level distance=0.5cm}, sibling distance=0.6cm,baseline=1.5ex]
\node [] {} [grow=up]	 
{child  {edge from parent[draw=none]  child {node [sq] (A)  {}  }}
};
\end{tikzpicture}
, ~~
\chi_2(\bsq) =  -\frac{1}{2}
\begin{tikzpicture}
[level 1/.style={level distance=0cm,sibling distance=.5cm }, level 2/.style={level distance=0.3cm,sibling distance=1.0cm }, level 3/.style={sibling distance=0.6cm,level distance=0.4cm}, sibling distance=0.6cm,baseline=2.5ex]
\node [] {} [grow=up]	 
{child  {edge from parent[draw=none]  child {node [sq] (A)  {} child {node [sq] (B) {} }		}}
};
\end{tikzpicture}
,~~
\chi_3(\bsq) = 
\frac{1}{3}
\begin{tikzpicture}
[level 1/.style={level distance=0cm,sibling distance=.5cm }, level 2/.style={level distance=0.3cm,sibling distance=1.0cm }, level 3/.style={sibling distance=0.6cm,level distance=0.4cm}, sibling distance=0.6cm,baseline=2.5ex]
\node [] {} [grow=up]	 
{child  {edge from parent[draw=none]  child {node [sq] (A)  {} child {node [sq] (B) {} child {node [sq] {}}  }		}}
};
\end{tikzpicture}
+
\frac{1}{12}
\begin{tikzpicture}
[level 1/.style={level distance=0cm,sibling distance=.5cm }, level 2/.style={level distance=0.3cm,sibling distance=1.0cm }, level 3/.style={sibling distance=0.6cm,level distance=0.4cm}, sibling distance=0.6cm,baseline=2.5ex]
\node [] {} [grow=up]	 
{	child  { edge from parent[draw=none]  child {node [sq] (A)  {} 	child {node [sq]  (B) {} 		} 		child {node [sq] (C) {}} 	 	}}
};
\end{tikzpicture}
+
\frac{1}{12}
\left(
\begin{tikzpicture}
[level 1/.style={level distance=0cm,sibling distance=.5cm }, level 2/.style={level distance=0.3cm,sibling distance=1.0cm }, level 3/.style={sibling distance=0.6cm,level distance=0.4cm}, sibling distance=0.6cm,baseline=2.5ex]
\node [] {} [grow=up]	 
{child  {edge from parent[draw=none]  child {node [sq] (A)  {} 		}}
	child  {edge from parent[draw=none]  child {{node [sq] (l1) {} child {node [sq] {}} } }    }
};
\end{tikzpicture}
-
\begin{tikzpicture}
[level 1/.style={level distance=0cm,sibling distance=.5cm }, level 2/.style={level distance=0.3cm,sibling distance=1.0cm }, level 3/.style={sibling distance=0.6cm,level distance=0.4cm}, sibling distance=0.6cm,baseline=2.5ex]
\node [] {} [grow=up]	 
{child  {edge from parent[draw=none]  child {node [sq] (A)  {} 	child {node [sq]  (B) {} 		} 		}}
	child {edge from parent[draw=none] child  {	 {node [sq,black] (l1) {} 
			}
	}}
};
\end{tikzpicture}
\right) 
\text{ and } 
\\
\chi_4(\bsq) &=
-\frac{1}{4}
\begin{tikzpicture}
[level 1/.style={level distance=0cm,sibling distance=.5cm }, level 2/.style={level distance=0.3cm,sibling distance=1.0cm }, level 3/.style={sibling distance=0.6cm,level distance=0.4cm}, sibling distance=0.6cm,baseline=2.5ex]
\node [] {} [grow=up]	 
{child  {edge from parent[draw=none]  child {node [sq] (A)  {} child {node [sq] (B) {} child {node [sq] {} child {node [sq] {}} }  }		}}
};
\end{tikzpicture}
-
\frac{1}{12}
\begin{tikzpicture}
[level 1/.style={level distance=0cm,sibling distance=.5cm }, level 2/.style={level distance=0.3cm,sibling distance=1.0cm }, level 3/.style={sibling distance=0.6cm,level distance=0.4cm}, sibling distance=0.6cm,baseline=2.5ex]
\node [] {} [grow=up]	 
{	child  {edge from parent[draw=none]  
		child {node [sq] {}	child {node [sq] (A)  {} 	
				child {node [sq] {} 	} 	 child {node [sq] {} } 	 	
		} }
	}
};
\end{tikzpicture}
-
\frac{1}{12}
\begin{tikzpicture}
[level 1/.style={level distance=0cm,sibling distance=.5cm }, level 2/.style={level distance=0.3cm,sibling distance=1.0cm }, level 3/.style={sibling distance=0.6cm,level distance=0.4cm}, sibling distance=0.6cm,baseline=2.5ex]
\node [] {} [grow=up]	 
{	child  { edge from parent[draw=none]  child {node [sq] (A)  {} 	child {node [sq]  (B) {} 		} 		
			child {node [sq] (C) {} child {node [sq] {}} } 	 	}}
};
\end{tikzpicture}
+
\frac{1}{24}
\left(
\begin{tikzpicture}
[level 1/.style={level distance=0cm,sibling distance=.5cm }, level 2/.style={level distance=0.3cm,sibling distance=1.0cm }, level 3/.style={sibling distance=0.6cm,level distance=0.4cm}, sibling distance=0.6cm,baseline=2.5ex]
\node [] {} [grow=up]	 
{child  {edge from parent[draw=none]  child {node [sq] (A)  {} 	child {node [sq]  (B) {} }  child {node [sq] {}} }  } 
	child {edge from parent[draw=none] child  {	 {node [sq,black] (l1) {} 
			}
	}}
};
\end{tikzpicture}
-
\begin{tikzpicture}
[level 1/.style={level distance=0cm,sibling distance=.5cm }, level 2/.style={level distance=0.3cm,sibling distance=1.0cm }, level 3/.style={sibling distance=0.6cm,level distance=0.4cm}, sibling distance=0.6cm,baseline=2.5ex]
\node [] {} [grow=up]	 
{child  {edge from parent[draw=none]  child {node [sq] (A)  {} 		}}
	child  {edge from parent[draw=none]  child {{node [sq] (l1) {} child {node [sq] {}  } child {node [sq] {}} } }    }
};
\end{tikzpicture}
\right) 
+
\frac{1}{12}
\left(
\begin{tikzpicture}
[level 1/.style={level distance=0cm,sibling distance=.5cm }, level 2/.style={level distance=0.3cm,sibling distance=1.0cm }, level 3/.style={sibling distance=0.6cm,level distance=0.4cm}, sibling distance=0.6cm,baseline=2.5ex]
\node [] {} [grow=up]	 
{child  {edge from parent[draw=none]  child {node [sq] (A)  {} 	child {node [sq]  (B) {}  child {node [sq] {} }		} 		}}
	child {edge from parent[draw=none] child  {	 {node [sq,black] (l1) {} 
			}
	}}
};
\end{tikzpicture}
-
\begin{tikzpicture}
[level 1/.style={level distance=0cm,sibling distance=.5cm }, level 2/.style={level distance=0.3cm,sibling distance=1.0cm }, level 3/.style={sibling distance=0.6cm,level distance=0.4cm}, sibling distance=0.6cm,baseline=2.5ex]
\node [] {} [grow=up]	 
{child  {edge from parent[draw=none]  child {node [sq] (A)  {} 		}}
	child  {edge from parent[draw=none]  child {{node [sq] (l1) {} child {node [sq] {} child {node [sq] {}} } } }    }
};
\end{tikzpicture}
\right).
\end{align*}

\section*{Acknowledgments}
The second author was supported by grant ``\#2018/19603-0, S\~ao Paulo Research Foundation (FAPESP)''. 

\bibliographystyle{abbrv}

\end{document}